\documentclass{article}
\usepackage{algorithm}
\usepackage{algorithmic}

\usepackage{graphicx}
\usepackage{rotating}

\usepackage{color}
\usepackage{bm}
\usepackage{amsthm}
\usepackage[usenames,dvipsnames,svgnames,table]{xcolor}

\usepackage{amssymb, amsfonts}
\usepackage{latexsym}
\usepackage{bm}
\usepackage{color}
\usepackage{array}
\usepackage{mathrsfs}
\usepackage{amsmath}
\usepackage{flafter}
\usepackage{multirow}
\usepackage{verbatim}

\definecolor{red}{rgb}{1,0,0}

\def\u{{\bm u}}

\def\ep{{\bm \epsilon}}
\def\sig{{\bm \sigma}}

\def\bU{{\boldsymbol U}}
\def\bV{{\boldsymbol V}}
\def\bu{{\boldsymbol u}}
\def\bv{{\boldsymbol v}}
\def\C{\Lambda}
\renewcommand{\div}{\operatorname{div}}
\newcommand{\Lbrack}{\lbrack\!\lbrack} 
\newcommand{\Rbrack}{\rbrack\!\rbrack} 

\newcommand{\vek}[1]{\boldsymbol{#1}}

\newtheorem{lemma}{Lemma}
\newtheorem{theorem}{Theorem}

\newtheorem{corollary}[theorem]{Corollary}
\newtheorem{proposition}[theorem]{Proposition}

\newtheorem{remark}[theorem]{Remark}

\usepackage{epsfig}
\usepackage{subfigure}

\def\Xi{{X}_h}

\def\ep{\bm \epsilon}
\def\sig{\bm \sigma}
\def\bU{\boldsymbol U}
\def\bV{\boldsymbol V}
\def\bP{\boldsymbol P}
\def\bf{\boldsymbol f}
\def\bu{\boldsymbol u}
\def\bv{\boldsymbol v}
\def\bp{\boldsymbol p}
\def\bw{\boldsymbol w}
\def\bo{\boldsymbol 0}
\def\bx{\boldsymbol x}
\def\bz{\boldsymbol z}
\def\bq{\boldsymbol q}
\def\bn{\boldsymbol n}
\def\divv{\text{div }}
\def\divn{\text{div }}
\def\Divv{\text{Div }}

\renewcommand{\div}{\operatorname{div}}
\title{Conservative discretizations and parameter-robust preconditioners for
Biot and multiple-network
flux-based
poroelastic models}
\author{Qinggou Hong, Johannes Kraus, Maria Lymbery, Fadi Philo}
\begin{document}
	%\begin{titlepage}
		
		\maketitle
		
	%\end{titlepage}
	%\begin{titlepage}

		%\tableofcontents
	%\end{titlepage}
	
\begin{abstract}
The parameters in the governing system of partial differential equations of multicompartmental poroelastic models typically
vary over several orders of magnitude making its stable discretization and efficient solution a challenging task. In this paper,
inspired  by the approach recently presented by Hong and Kraus~[Parameter-robust stability of classical three-field formulation
of Biot's consolidation model, ETNA (to appear)] for the Biot model, we prove the uniform stability, and design stable
disretizations and parameter-robust preconditioners for flux-based formulations of multiple-network poroelastic systems. 
Novel parameter-matrix-dependent norms that provide the key for establishing uniform inf-sup stability of the continuous
problem are introduced. As a result, the stability estimates presented here are uniform not only with
respect to the Lam\'e parameter $\lambda$, but also with respect to all the other model parameters such as permeability
coefficients $K_i$, storage coefficients $c_{p_i}$, network transfer coefficients $\beta_{ij}, i,j=1,\cdots,n$, the scale of the
networks $n$ and the time step size $\tau$.

Moreover, strongly mass conservative discretizations that meet the required conditions for parameter-robust stability are suggested
and corresponding optimal error estimates proved. The transfer of the canonical (norm-equivalent) operator preconditioners from
the continuous to the discrete level lays the foundation for optimal and fully robust iterative solution methods. The theoretical
results are confirmed in numerical experiments that are motivated by practical applications.
\end{abstract}


%\begin{AMS}65F10, 65N20, 65N30\end{AMS}
	
\section{Introduction}	

Multiple-network poroelastic theory (MPET) has been introduced into geomechanics~\cite{Bai_etal1993multi} to describe mechanical
deformation and fluid flow in porous media as a generalization of Biot's theory~\cite{Biot1941general,Biot1955theory}.
The deformable elastic matrix is assumed to be permeated by multiple fluid networks of pores and fissures with differing porosity and
permeability.

During the last decade, MPET has acquired many important applications in medicine and biomechanics and therefore become an active
area of scientific research. The biological MPET model captures flow across scales and networks in soft tissue and can be used as an
embedding platform for more specific models, e.g. to describe water transport in the cerebral environment and to explore hypotheses
defining the initiation and progression of both acute and chronic hydrocephalus~\cite{TullyVentikos2011cerebral}.
In~\cite{Vardakis2013exploring,Vardakis2016investigating} multicompartmental poroelastic models have been proposed to study the
effects of obstructing cerebrospinal fluid (CSF) transport within an anatomically accurate cerebral environment and to demonstrate
the impact of aqueductal stenosis and fourth ventricle outlet obstruction (FVOO). As a consequence, the efficacy of treating such clinical
conditions by surgical procedures that focus on relieving the buildup of CSF pressure in the brain's third or fourth ventricle could be explored
by means of computer simulations, which can also assist in finding medical indications of oedema formation~\cite{Chou2016afully}.

Recently, the MPET model has also been used to better understand the influence of biomechanical risk factors associated
with the early stages of Alzheimer's disease (AD), the most common form of dementia~\cite{Guo_etal2018subject-specific}.
Modeling transport of fluid within the brain is essential in order to discover the underlying mechanisms that are currently being investigated
with regard to AD, such as the amyloid hypothesis according to which the accumulation of neurotoxic amyloid-$\beta$ (A$\beta$) into
parenchymal senile plaques or within the walls of arteries is a root cause of this disease. 

Biot's and multiple-network poroelastic models are challenging from a computational point of view in that the physical parameters for
different practical applications exhibit extremely large variations. For instance, permeabilities in geophysical applications typically
range from $10^{-9}$ to $10^{-21} m^2$ while Young's modulus is of the order of GPa and the Poisson ratio in the range~$0.1-0.3$,
see~\cite{Wang2000theory,Lee2016parameter,Coussy2004poromechanics}.
Permeabilities in biological applications typically range from $10^{-14}$ to $10^{-16} m^2$. Young's modulus of soft tissues is in the
order of kPa and the Poisson ratio in the range $0.3$ to almost $0.5$, see, e.g.,~\cite{Smith2007interstitial,Stoverud2016poroelastic}.
For that reason it is important that the problem is well posed and the numerical methods for its solution are stable over the whole range
of values of the physical (model) and discretization parameters.

The stability of the time discretization and space discretization by finite difference or finite volume methods have been 
studied in~\cite{Axelsson2012stable, gaspar2003finite, gaspar2006staggered, nordbotten2016stable} and will not be
addressed here. Instead we focus on the issue of uniform inf-sup stable finite element discretizations of the static multiple-network poroelastic
problem. 
It is well known that the well-posedness analysis of saddle-point problems in their weak formulation, apart from the
boundedness and definiteness of the underlying bilinear form, relies on a stability estimate that is often referred
to as Ladyzenskaja-Babuska-Brezzi (LBB) condition~\cite{Boffi2013mixed,Ern2004theory}.
The LBB condition, see~\cite{Babuska1971error,Brezzi1974existence}, is also crucial in the analysis of stable
discretizations and in the derivation of a priori error etsimates for mixed problems. Inf-sup stability for the Darcy
problem as well as for the Stokes and linear elasticity problems have been established under rather general conditions
and various stable mixed discretizations of either of these problems have been proposed over the years,
see, e.g.~\cite{Boffi2013mixed} and the references therein.

Biot's model of poroelasticity combines these equations and the parameter-robust stability of its classical three-field formulation
has been established only recently in~\cite{HongKraus2017parameter}. Alternative formulations that can be proven to
be stable include a two-field formulation for the displacement
and the pore pressure~{\cite{boffi2016nonconforming, adler2017robust}} and a new three-field formulation based on introducing
the total pressure as a weighted sum of fluid and solid pressure as the third unknown besides the displacement and fluid
pressure~\cite{oyarzua2016locking, Lee2016parameter}.
Contrary to this new three-field formulation as analyzed in~\cite{Lee2016parameter}, the classic three-field
formulation of Biot's consolidation model considered in~\cite{HongKraus2017parameter} builds on Darcy's law
in order to guarantee fluid mass conservation, a property that the discrete models studied in this paper maintain.  
Aside from two- and three-field formulations, a four-field formulation has been considered for the Biot model in which
the stress tensor is kept as a variable in the system, see~\cite{Lee2016robust}. The error analysis in the latter work is
robust with respect to the Lam\'e parameter $\lambda$, but not uniform with respect to the other model parameters
such as $K$.
Another formulation for Biot's model has recently been proposed and analyzed in~\cite{Baerland2017weakly}. The authors
use  mixed methods based on the Hellinger-Reissner variational principle for the elasticity part of the system, and impose
weakly the symmetry of the stress tensor $\sig$, resulting in a saddle point problem for $\sig$, $\u$, $p$, and a Lagrange
multiplier. They prove the parameter-robust stability of the resulting four-field formulation.

The first attempt to design parameter-robust discretizations and analyze their stability for the MPET model
is presented in~\cite{lee2018mixed}. Motivated by~\cite{oyarzua2016locking, Lee2016parameter}, the
authors of~\cite{lee2018mixed} propose a mixed finite element formulation based on introducing an additional total
pressure variable. Utilizing energy estimates for the solutions of the continuous problem and a priori error
estimates for a family of compatible semi-discretizations, they show that the formulation is robust in the limits of
incompressibility, vanishing storage coefficients, and vanishing transfer between networks. %in \cite{lee2018mixed}.
The robustness with respect to the permeability coefficients remains an open question in~\cite{lee2018mixed}.

There are various discretizations for the {classic three-field} formulation of Biot's model that meet the conditions for the proof
of full parameter-robust stability that has been presented in~\cite{HongKraus2017parameter}. In general, whenever
a discretization is based on a Stokes-stable pair of finite element spaces for the displacement and pressure and a
Poisson-stable pair of finite element spaces for the flux and pressure unknowns, it is possible to define a
parameter-dependent norm (which in general is not uniquely determined) such that the constant in the inf-sup
condition for the Biot problem does not depend on any of the model or disretization parameters.
For example, the triplets $CR_{l}/RT_{l-1}/P_{l-1} (l=1,2)$ together with the stabilization techniques suggested
in~\cite{hansbo2003discontinuous,Hu2017nonconforming}, see also~\cite{fortin1983non}, or the triplets
$P_2/RT_0/P_0$ (in 2D) and $P_2^{stab}/RT_0/P_0$ (in 3D), or $P_2/RT_1/P_1$, or the stabilized discretization
that has recently been advocated in~\cite{rodrigo2017new}, or the finite element methods proposed in~\cite{lee2018robust}
would qualify for such parameter-robustness. However, the above-mentioned finite element methods do not have the property of 
strong mass conservation in the sense of satisfying the mass balance equation pointwise and therefore locally and globally on a discrete level.

A priori error estimates for the
continuous-in-time scheme and the discontinuous Galerkin (DG) spatial discretization (similar to \cite{HongKraus2017parameter})
have been presented in~\cite{KanschatRiviere2017finite} for the Biot model.
Inspired by the approach proposed in~\cite{HongKraus2017parameter} in context of the static Biot problem, we make use of the
DG technology in the present  work for solving the MPET system by introducing novel parameter-matrix-dependent norms.

The aim of this work is to establish the results regarding the parameter-robust
stability of the weak formulation of the continuous problem as well as the stability of strongly mass conservative
discretizations, corresponding error estimates and parameter-robust preconditioners for the multiple-network (MPET) model. 
The presented stability results and error estimates and preconditioners are independent of {all model and discretization
parameters} including the Lam\'e parameter~$\lambda$, permeability coefficients $K_i$, arbitrary small or even vanishing
storage coefficients $c_{p_i}$, network transfer coefficients $\beta_{ij}, i,j=1,\cdots,n$, the scale of the networks $n$, the time
step size $\tau$ and mesh size $h$. To our knowledge, these are the first fully parameter-robust stability results for the
MPET model in a flux-based formulation.

The paper is organized as follows. In Section~\ref{sec:ModelProblem} the multiple-network poroelastic model is stated
in a flux-based formulation, which can be considered as an extension of the classical three-field formulation considered
in~\cite{HongKraus2017parameter}. The governing partial differential equations are then rescaled and the static
boundary-value problem resulting from semi-discretization in time by the implicit Euler method is presented in its
weak formulation in the beginning of Section~\ref{sec:par_rob_stab_model}. The proofs of the uniform boundedness and
the parameter-robust inf-sup stability of the underlying bilinear form are the main results that follow in this section.
Section~\ref{sec:uni_stab_disc_model} then discusses a class of uniformly stable and strongly mass conservative
mixed finite element discretizations that are based on $H(\div)$-conforming discontinuous Galerkin approximations of
the displacement field. Uniform boundedness and inf-sup stability are proved to be independent of all model and
discretization parameters and the corresponding parameter-robust preconditioners are provided. 
Next, in Section~\ref{sec:error_estimates}, optimal parameter-robust error estimates are proved.
Finally, Section~\ref{sec:NumericalExperiments} is devoted to the validation and illustration of the theoretical results
in this work and Section~\ref{conclusion} provides a brief conclusion.

\section{Model Problem}\label{sec:ModelProblem}

In an open domain
$\Omega \subset \mathbb R^d$, $d=2,3$, the unknown physical variables in the MPET flux based model are the displacement
$\bu$, fluxes $\bv_i$ and corresponding pressures $p_i$ $i=1,\dots,n$. The equations describing the model are as follows:
%The MPET equations describe displacement $\bu$ and $n$-fluxes $\bv_i$ and n-corresponding pressures $p_i$ \\ The equations describing the model are 
\begin{subequations}
\begin{align}
-\divv \sig + \sum_{i=1}^{n}\alpha_i \nabla p_i  &= \bf~~ \text{in}~~ \Omega\times (0,T),\label{MPET1}\\
\bv_i &= -K_i \nabla p_i\;\; \text{in}~~ \Omega\times (0,T),~~ i=1,\ldots,n,  \label{MPET2} \\
-\alpha_i \divv \dot{\bu} - \divv \bv_i - c_{p_{i}} \dot{p}_{i} - \sum_{\substack{j=1\\j\neq i}}^{n}\beta_{ij} (p_i-p_j) &=g_i\;\;\text{in}~~ \Omega\times (0,T),~~  i=1,\ldots,n,\label{MPET3}
		\end{align}
		\label{eq:MPET}
\end{subequations}
where
\begin{subequations}
	\begin{align}
	\sig &= 2\mu \ep(\bu) + \lambda \text{div}(\bu)\bm I,\label{constitutive}\\
	\ep(\bu) &= \frac{1}{2}(\nabla \bu + (\nabla \bu)^T). \label{compatibility}	
	\end{align}
\end{subequations}

In equation~\eqref{constitutive}, $\lambda$ and $\mu$ denote the Lam\'e parameters defined in terms of the modulus of
elasticity (Young's modulus) $E$ and the Poisson ratio $\nu\in[0,1/2)$ by
$$
\lambda:=\frac{\nu E}{(1+\nu)(1-2\nu)}, \quad
\mu:=\frac{E}{2(1+\nu)}.
$$

The constants $\alpha_i$ appearing in~\eqref{MPET1} couple $n$ pore pressures $p_i$ with the displacement variable $\bu$ 
and are known in the literature as Biot-Willis parameters. The corresponding right hand side $\bf$ describes the body force density. 
Each fluid flux $\bv_i$ is related to a specific negative pressure gradient $-\nabla p_i$ via Darcy's law in~\eqref{MPET2}. 
The tensors $K_i$ denote the hydraulic conductivities
%, i.e., the ratios of fluid velocities to hydraulic gradients 
which give an indication of the general permeability of a porous medium. 
In~\eqref{MPET3} $\dot{\bu}$ and $\dot p_i$ express the time derivatives of  the displacement $\bu$ and the pressure variables $p_i$. 
The constants $c_{p_{i}}$ are referred to as the constrained specific storage coefficients and are connnected to compressibility
of each fluid, for more see e.g.~\cite{Showalter2010poroelastic} and the references therein. 
The parameters $\beta_{ij}$ are the network transfer coefficients coupling the network {pressures \cite{TullyVentikos2011cerebral}},
hence $\beta_{ij}=\beta_{ji}$. The source terms $g_i$ in~\eqref{MPET3} represent forced fluid extractions or injections into the medium.

It is assumed that the effective stress tensor $\sig$ satisfies Hooke's law~\eqref{constitutive} where the effective strain tensor
$\ep(\bu)$ is given by the symmetric part of the gradient of the displacement field, see~\eqref{compatibility}. 
Here $\boldsymbol I$ is used to denote the identity tensor. 

The following boundary and initial conditions guarantee the well posedness of system~\eqref{eq:MPET}: 
\begin{subequations}\label{eq:MPET_BC}
\begin{eqnarray}
p_i(\bx,t) &=& p_{i,D}(\bx,t)  \qquad \mbox{for } \bx \in \Gamma_{p_i,D}, \quad t > 0, \quad i=1,\ldots,n,\\ 
\bv_i(\bx,t) \cdot {\bm n} (\bx) &=& q_{i,N}(\bx,t)  \qquad \mbox{for }  \bx \in \Gamma_{p_i,N}, \quad t > 0 ,\quad i=1,\ldots,n,\\ 
\bu(\bx,t) &=& {\bu}_D(\bx,t)  \;\qquad \mbox{for }  \bx \in \Gamma_{\bu,D}, \,\quad t > 0, \\ 
({\sig(\bx,t)}-\sum_{i=1}^{n}\alpha_i p_i  \bm I) \, {\bn} (\bx) &=& {\bm g}_N(\bx,t) \;\, \qquad \mbox{for }  \bx \in \Gamma_{\bu,N}, \,\quad t > 0,
\end{eqnarray}
%\begin{eqnarray}
%p(\boldsymbol x,t)=p, &\quad \sig \boldsymbol n &= \boldsymbol h \qquad \mbox{on } \Gamma_1, \\
%\u = \boldsymbol 0, &\qquad K (\nabla p) \cdot \boldsymbol n &= 0  \, \qquad \mbox{on } \Gamma_2,
%\end{eqnarray}
\end{subequations}
where for $i=1,\ldots,n$ it is fulfilled 
$\Gamma_{p_i,D} \cap \Gamma_{p_i,N} = \emptyset$,
$\overline{\Gamma}_{p_i,D}\cup \overline{\Gamma}_{p_i,N}=\Gamma=\partial{\Omega}$
and
$\Gamma_{\bu,D} \cap \Gamma_{\bu,N} = \emptyset$,
$\overline{\Gamma}_{\bu,D} \cup \overline{\Gamma}_{\bu,N}=\Gamma$.
{Initial conditions at the time $t=0$ to complement the boundary conditions~\eqref{eq:MPET_BC},  have to satisfy~\eqref{MPET1}, and
are given by}
\begin{subequations}\label{eq:Biot_IC}
\begin{eqnarray}
p_i(\bx,0) &=& p_{i,0}(\bx) \qquad \bx \in \Omega,\quad i=1,\ldots,n, \\
\bu (\bx,0) &=& \bu_0(\bx) \qquad\;\, \bx \in \Omega.
\end{eqnarray}
\end{subequations}

The stress variable $\sig$ is eliminated from the MPET system by substituting the constitutive
equation~\eqref{constitutive} in~\eqref{MPET1} thus obtaining the classical flux-based formulation of the MPET model.

To solve numerically the time-dependent problem, the backward Euler method is employed for time discretization
resulting in the following system of time-step equations:

\begin{equation}\label{eq:1}
\mathcal{A}
\begin{bmatrix}
\bu^k \\
\bv_1^k\\
\vdots\\
\bv_n^k\\
p_1^k\\
\vdots\\
p_n^k\\
\end{bmatrix}
=
\begin{bmatrix}
\bf^k \\
\bo\\
\vdots\\
\bo\\
{g}_1^k\\
\vdots\\
{g}_n^k\\
\end{bmatrix},
%\label{eq:backwardEuler1}
\end{equation}
where
\begin{equation}\label{eq:2}
\mathcal{A}:=
\begin{bmatrix}
-2 \mu \divv \ep  - \lambda \nabla \divv     & 0  & \dots     & \dots&0      &  \alpha_1 \nabla & \dots & \dots &\alpha_n \nabla \\
\\
0      &\tau K^{-1}_1I &     0      & \dots  &0         & \tau \nabla  &  0      & \dots  &0 \\
\vdots &  0           &     \ddots &     & \vdots       &  0           &     \ddots &     & \vdots \\
\vdots &      \vdots  &            &\ddots  & 0            &       \vdots  &            &\ddots  & 0 \\
0      &     0        &  \dots     & 0      & \tau K^{-1}_nI&  0        &  \dots     & 0   &\tau \nabla\\
\\
-\alpha_1 \divv &- \tau \divv&0 &\dots   &  0            & \tau\tilde {\beta}_{11}I &\tau \beta_{12}I & \dots& \tau \beta_{1n}I\\
\vdots                 &   0       &   \ddots        &                 & \vdots      & \tau \beta_{21}I & \ddots & &{\tau \beta_{2n}}I\\
\vdots                 &  \vdots  &              & \ddots              & 0    &  \vdots &  & \ddots & \vdots\\
-\alpha_n \divv &   0    & \dots &   0    &- \tau \divv  &\tau \beta_{n1}I  &  { \tau \beta_{n2} } I &   \dots   &\tau \tilde {\beta}_{nn}I \\     
\end{bmatrix},
\end{equation}
$$
\tilde {\beta}_{ii}=- \frac{c_{p_i}}{\tau} - \beta_{ii}, ~\hbox{and}~\beta_{ii} = \sum_{\substack{j=1\\j\neq i}}^{n}\beta_{ij},~ i=1,\dots, n. 
$$
The unknown time-step functions $\bu^k$, $\bv_i^k$, $p_i^k$ for $i=1,\ldots,n$ at any given time $t = t_k = t_{k-1} + \tau$ are defined as
$$
\begin{array}{cccl}
\bu^k &=& \bu(x,t_k) \in \bU &:= \{ \bu\in H^1(\Omega)^d : \bu = \bu_D \text{ on } \Gamma_{\bu,D} \},\\
\bv_i^k &=&\bv_i(x,t_k) \in \bm V_i &:= \{ \bv_i\in H(\divv,\Omega) :\bv_i\cdot \bm{n} = q_{i,N} \text{ on } \Gamma_{p_i,N} \},\\
p_i^k &=& p_i(x,t_k) \in P_i &:= L^2(\Omega),
\end{array}
$$
whereas the right hand side time-step functions are $\bf^k = \bf(x, t_k)$, ${g}^k_i = -\tau g_i(x, t_k) - \alpha_i \divv(\bu^{k-1}) - c_{p_i}p_i^{k-1}$, $i=1,\ldots,n$. 
Later, the static problem~\eqref{eq:1}--\eqref{eq:2} is considered and, for convenience, the superscript for the time-step functions is dropped, that is, 
$\bu^k,\bv_i^k$ and $p_i^k$ will be denoted by $\bu,\bv_i$ and $p_i$, respectively.

The considered function spaces are as follows: 
\begin{itemize}
\item $L^2(\Omega)$ is the space of square Lebesgue
integrable functions equipped with the standard $L^2$ norm $\|\cdot\|$;
\item $H^1(\Omega)^d$ denotes the
space of vector-valued $H^1$-functions equipped with the norm $\| \cdot \|_1$ for which 
$\|\bm u\|^2_1:=\|\bm u\|^2+\|\nabla \bm u\|^2$;
\item $
H(\div ;\Omega):=\{\bm v \in L^2(\Omega)^d: \div  \bm v\in L^2(\Omega)\}
$
with norm $\|\cdot\|_{\div}$ defined by
$\|\bm v\|^2_{\div}:=\|\bm v\|^2+\|\div \bm v\|^2$.
\end{itemize}

When the case $\Gamma_{\bu,D}=\Gamma_{p_i,N}=\Gamma$ and ${\bu}_D=\bo$, $q_{i,N}=0$ is considered, the notations 
$\bU=H^1_0(\Omega)^d$ and  $\bm V_i=H_0(\div, \Omega)$, $i=1,\ldots,n$ are used. To guarantee the uniqueness of the 
solution for the pressure variables $p_i$, we set $P_i= L^2_0(\Omega):=\{ p \in L^2(\Omega) : \int_{\Omega} p \, d\bx = 0 \}$ for
$i=1,\ldots,n$.

\section{Stability analysis} \label{sec:par_rob_stab_model}
First the parameter $\mu$ is eliminated from the system by dividing equations~\eqref{eq:1}--\eqref{eq:2} by $2 \mu$ and
making the substitutions:
{
$$
2\mu \rightarrow 1, \frac{\lambda}{2\mu}\rightarrow \lambda, \frac{\alpha_i}{2\mu}\rightarrow \alpha_{i}, \frac{\bm f}{2\mu}\rightarrow \bm f, \frac{\tau}{2\mu}\rightarrow \tau, 
\frac{c_{p_i}}{2\mu}\rightarrow c_{p_i}, \frac{g_i}{2\mu} \rightarrow g_i,~ \text{for}~ i=1,\ldots,n.
$$
}
Equation~\eqref{eq:1} then becomes
\begin{subequations}\label{eq:3}
	\begin{align}
	-\divv \ep(\bu)  - \lambda \nabla \divv \bu +\sum_{i=1}^{n}\alpha_i \nabla p_i &= \bf, \label{eq:3a}  \\
	\tau K^{-1}_i\bv_i + \tau \nabla p_i  &= \bo, \qquad i=1,\dots,n,  \label{eq:3b}\\
	-\alpha_i \divv\bu - \tau \divv\bv_i -c_{p_i}p_i -\tau \sum_{\substack{{j=1}\\j\neq i}}^{n}\beta_{ij}(p_i -p_j) &= g_i, \qquad i=1,\dots,n. \label{eq:3c}
	\end{align}
\end{subequations}
Next, equation~\eqref{eq:3b} is multiplied by $\alpha_i \tau^{-1}$, equation~\eqref{eq:3c} {is multiplied} by $\alpha_i^{-1}$
so that the substitutions 
$$
 \tilde{\bv}_i:= \frac{\tau}{\alpha_i} \bv_i, \quad \tilde{p}_i: =\alpha_i p_i, \quad \tilde{g}_i:=\frac{g_i}{\alpha_i}
$$
yield
\begin{subequations}\label{eq:4}
	\begin{align}
	-\divv \ep(\bu) - \lambda \nabla \divv \bu +\sum_{i=1}^{n} \nabla \tilde{p}_i &= \bf,\\
	\tau^{-1} K^{-1}_i\alpha_i^2\tilde{\bv}_i + \nabla \tilde{p}_i &= \bo, \qquad i=1,\dots,n,\\
	-\divv \bu - \divv\tilde{\bv}_i  -\frac{c_{p_i}}{\alpha_i^2} \tilde{p}_i + \sum_{\substack{{j=1}\\j\neq i}}^{n} \left(-\frac{\tau\beta_{ij}}{\alpha_i^2} \tilde{p}_i+\frac{\tau\beta_{ij}}{\alpha_i\alpha_j} \tilde{p}_j\right) 
	&= \tilde{g}_i,\qquad i=1,\dots,n.
	\end{align}
\end{subequations}
For convenience, { the ``tilde" symbol} is skipped and system~\eqref{eq:4} is written as:
\begin{subequations}\label{eq:5}
	\begin{align}
	-\divv \ep(\bu) - \lambda \nabla \divv \bu +\sum_{i=1}^{n} \nabla p_i &= \bf,\\
	\tau^{-1} K^{-1}_i\alpha_i^2\bv_i + \nabla p_i &= \bo, \qquad i=1,\dots,n,\\
	-\divv \bu - \divv\bv_i  -\frac{c_{p_i}}{\alpha_i^2} p_i + \sum_{\substack{{j=1}\\j\neq i}}^{n} \left(-\frac{\tau\beta_{ij}}{\alpha_i^2} p_i+\frac{\tau\beta_{ij}}{\alpha_i\alpha_j} p_j\right) &= g_i,\qquad i=1,\dots,n.
	\end{align}
\end{subequations}
Further, we denote
\begin{align*}
R^{-1}_i = \tau^{-1}K_i^{-1}\alpha_i^2, \quad
%b_{ij} = \frac{\tau \beta_{ij}}{\alpha_i \alpha_j}  \text{ , }
\alpha_{p_i} = \frac{c_{p_i}}{\alpha_i^2}, \quad 
{\alpha}_{ij} = \frac{\tau \beta_{ij}}{\alpha_i \alpha_j},~ i,j=1,\cdots,n,
\end{align*}
and make the rather general and reasonable assumptions that
\begin{align*}
{ \lambda > 0},\quad  R^{-1}_i > 0,\;\; \alpha_{p_i} \geq 0\;\;\text{for}\;\; i=1,\ldots,n,\quad \text{and}\quad {\alpha}_{ij}\geq 0\;\; \text{for}\;\;i,j=1,\ldots,n.
\end{align*}
Making use of these substitutions, without loss of generality, system~\eqref{eq:1} becomes
\begin{subequations}\label{eq:6}
	\begin{align}
	-\divv \ep(\bu) - \lambda \nabla \divv \bu +\sum_{i=1}^{n} \nabla p_i &= \bf,\label{eq:6,1}\\
	R^{-1}_i\bv_i + \nabla p_i &= \bo, \qquad i=1,\dots,n, \label{eq:6,2}\\
	-\divv \bu - \divv\bv_i  -(\alpha_{p_i} + {\alpha}_{ii}) p_i+\sum_{\substack{{j=1}\\j\neq i}}^{n}{\alpha}_{ij} p_j &= g_i,\qquad i=1,\dots,n, \label{eq:6,3}
	\end{align}
\end{subequations}
or %shortly written
\begin{equation}\label{eq:7}
\mathcal{A}
\begin{bmatrix}
\bu \\
\bv_1\\
\vdots\\
\bv_n\\
p_1\\
\vdots\\
p_n\\
\end{bmatrix}
=
\begin{bmatrix}
\bf \\
\bo\\
\vdots\\
\bo\\
{g}_1\\
\vdots\\
{g}_n\\
\end{bmatrix}
%\label{eq:backwardEuler1}
\end{equation}
where  

\begin{equation}\label{operator:A}
\mathcal{A}:=
\begin{bmatrix}
- \divv \ep  - \lambda \nabla \divv     & 0  & \dots     & \dots&0      &   \nabla & \dots & \dots & \nabla    \\
\\
0      & R_1^{-1}I &     0      & \dots  &0         &  \nabla  &  0      & \dots  &0 \\
\vdots &  0           &     \ddots &     & \vdots       &  0           &     \ddots &     & \vdots \\
\vdots &      \vdots  &            &\ddots  & 0            &       \vdots  &            &\ddots  & 0 \\
0      &     0        &  \dots     & 0      &R_n^{-1}I&  0        &  \dots     & 0   & \nabla\\
\\
- \divv &-  \divv&0 &\dots   &  0            &\tilde\alpha_{11}I& \alpha_{12} I & \dots& \alpha_{1n}I\\
\vdots                 &   0       &   \ddots        &                 & \vdots      &  \alpha_{21}I  & \ddots & & \alpha_{2n}I\\
\vdots                 &  \vdots  &              & \ddots              & 0    &  \vdots &  & \ddots & \vdots\\
- \divv &   0    & \dots &   0    &-  \divv  & \alpha_{n1}I  &  \alpha_{n2}I &   \dots   & \tilde\alpha_{nn} I  \\     
\end{bmatrix}
\end{equation}
is the scaled operator from~\eqref{eq:2} and  $\tilde\alpha_{ii}=- \alpha_{p_i} - \alpha_{ii}, i=1,\dots,n$.

%
%$\bp\in \bP \text{  where}\quad \bp=(p_1,\dots,p_n)^T\quad \bP=P_1\times \dots\times P_n$
%\\
%$\bv\in \bV \text{  where}\quad \bv=(\bv_1,\dots,\bv_n)^T\quad \bV=\bV_1\times \dots\times \bV_n$\\
For convenience, let $\bv^T=(\bv_1^T,\dots,\bv_n^T)$, $\bp^T=(p_1,\dots,p_n)$, $\bz^T=(\bz_1^T,\dots,\bz_n^T)$,
$\bq^T=(q_1,\dots,q_n)$ and  $\bV=\bV_1\times \dots\times \bV_n$, $\bP=P_1\times \dots\times P_n$. 
Taking into account the boundary conditions, system~\eqref{eq:6} has the following weak formulation: 
Find $(\bu; \bv;\bp) \in \bU\times \bV\times \bP$, 
such that for any $(\bw; \bz;\bq)  \in \bU\times \bV\times \bP$ there holds 
%Find $(\bu, \underbrace{(\bv_1,\dots,\bv_n)}_{\bv^T},\underbrace{(p_1,\dots,p_n)}_{\bp^T}) \in \bU\times\underbrace{ \bV_1\times \dots\times \bV_n}_{\bV}\times\underbrace{ P_1\times \dots\times P_n}_{\bP}$, 
%such that for any $(\bw, \underbrace{(\bz_1 ,\dots ,\bz_n)}_{\bz} , \underbrace{(q_1,\dots, q_n)}_{\bq})  \in \bU\times \underbrace{\bV_1\times \dots\times \bV_n}_{\bV}\times\underbrace{ P_1\times \dots\times P_n}_{\bP}$
\begin{subequations}\label{eq:8}
	\begin{align}
	(\ep(\bu), \ep(\bw))+ \lambda  (\divv\bu,\divv \bw) -\sum_{i=1}^{n} (p_i,\divv \bw) &= (\bf,\bw)\label{eq:8a}\\
	 (R^{-1}_i\bv_i,\bz_i) {-} (p_i,\divv \bz_i) &= 0, \quad i=1,\dots,n,\label{eq:8b}\\
	-(\divv\bu,q_i) - (\divv\bv_i,q_i)  -(\alpha_{p_i}+{\alpha}_{ii}) (p_i,q_i) + \sum_{\substack{{j=1}\\j\neq i}}^{n}  \alpha_{ij}(p_j,q_i) &= (g_i,q_i) ,\quad  i=1,\dots,n.\label{eq:8c}
	\end{align}
\end{subequations}
Following~\cite{lipnikov2002numerical}, we first consider the following Hilbert spaces and weighted norms
\begin{align}
\bU &= H_0^1(\Omega)^d, \qquad\qquad~(\bu,\bw)_{\bU} = (\ep(\bu),\ep(\bw)) + \lambda(\divv\bu,\divv\bw),\\
\bm V_i&= H_0({\rm div},\Omega), \qquad\quad(\bv_i,\bz_i)_{\bV_i} = (R_i^{-1}\bv_i,\bz_i) + (R_i^{-1}{\rm div}\bv_i,{\rm div}\bz_i),\qquad i=1,\dots,n,\\
P_i&= L_0^2(\Omega), \qquad\quad \quad\quad(p_i,q_i)_{P_i}=(p_i,q_i),\qquad i=1,\dots,n
\end{align}
%The natural norms for these spaces have the following definitions:
%\begin{subequations}
%	\begin{align}
%	(\bu,\bw)_U &= (\ep(\bu),\ep(\bw)) + \lambda(\divv\bu,\divv\bw)\\
%	(\bv_i,\bz_i)_{V_i} &= R_i^{-1}(\bv_i,\bz_i) + R_i^{-1}(\bv_i,\bz_i)\qquad\qquad i=1,\dots,n
%	\end{align}
%	\label{eq:norm1}
%\end{subequations}
System~\eqref{eq:8}, however, is not uniformly stable with respect to the parameters $R_i^{-1}$ under these norms as shown
in~\cite{HongKraus2017parameter}. 
Therefore, proper parameter-dependent norms for the spaces $\bU$, $\bm V_i$, $P_i$, $i=1,\ldots,n$, have to be introduced
that allow to establish the parameter-robust stability of the MPET model~\eqref{eq:8} for parameters in the ranges
\begin{align}\label{parameter:range}
{\lambda > 0,}\quad  R^{-1}_1 ,\dots, R^{-1}_n > 0 ,\quad \alpha_{p_1},\dots,\alpha_{p_n} \ge 0,\quad \alpha_{ij}\ge 0, ~~~i, j=1,\dots, n.
\end{align}
From experience, we know that the largest of the values $R^{-1}_i, i=1,\dots,n$ is important to us, and we note that the term
$(\ep(\bu),\ep(\bw))$ dominates in the elasticity form when $\lambda \ll 1$. Hence, we define
\begin{equation}
R^{-1} = \max\{ R_1^{-1} , \dots,R_n^{-1} \},~~\lambda_0=\max\{1, \lambda\}.
\end{equation}
Again by trial and error, we find that we have to deal with the parameters in a ``matrix" format. Therefore, we define the following
$n\times n$ matrices
\begin{align*}
\C _{1}&=  
\begin{bmatrix}
\alpha_{11} & -\alpha_{12} & \dots &-\alpha_{1n}  \\
-\alpha_{21} & \alpha_{22} & \dots &-\alpha_{2n}  \\
\vdots & \vdots & \ddots & \vdots  \\
-\alpha_{n1} & -\alpha_{n2} & \dots &\alpha_{nn}  
\end{bmatrix},\qquad
\C _2=
\begin{bmatrix}
{\alpha_{p_1}} &0&\dots &0\\
0&{\alpha_{p_2}}&\dots &0\\
\vdots&\vdots &\ddots&\vdots\\
0&0&\dots&  {\alpha_{p_n}}
\end{bmatrix},\\
 \C_{3}&=
\begin{bmatrix}
R &0&\dots &0\\
0&R&\dots &0\\
\vdots& \vdots&\ddots&\vdots\\
0&0&\dots & R
\end{bmatrix},\qquad
{ \C_{4}=
\begin{bmatrix}
\frac{1}{\lambda_0} &\dots& {\dots} & \frac{1}{\lambda_0}\\
\vdots& & &\vdots\\
\vdots& & &\vdots\\
\frac{1}{\lambda_0}&\dots & {\dots}&\frac{1}{\lambda_0}
\end{bmatrix}.
}
\end{align*}
From the definition of
$\alpha_{ij}=\frac{\tau\beta_{ij}}{\alpha_i\alpha_j}$, $\beta_{ii}= \sum\limits_{\substack{{j=1}\\j\neq i}}^{n}\beta_{ij}$
and $\beta_{ij}=\beta_{ji}$, it is obvious that $\C_1$ is symmetric positive semidefinite (SPSD).
Since $\alpha_{p_i}\ge 0$, we have that $\C_2$ is SPSD. Noting that $R>0$,
it follows that $\C_3$ is symmetric positive definite (SPD). Moreover, it is obvious that $\C_4$ is a rank-one
matrix with eigenvalues $\lambda_i =  0,~i=1,\dots,n-1 $ and $ \lambda_n = \frac{n}{\lambda}$. 
\begin{remark}
Let $\bm g^T=(g_1, \cdots,g_n), \bm g_c=\frac{1}{|\Omega|}\int_{\Omega}\bm g dx$ and $\C_g=[\C_1+\C_2, \bm g_c]$
be the matrix that is obtained by augmenting $\C_1+\C_2$ with the column $\bm g_c$. In general, we assume that
$\int_{\Omega}\bm g dx=\bm 0$.
When $\C_1+\C_2$ is the zero matrix, this assumption is a ``(classical) consistency condition''.
If $\C_1+\C_2$ is nonzero and $\int_{\Omega}\bm g dx\neq\bm 0$, then $\bm g$ has to satisfy the ``general consistency condition"
${\rm rank}(\C_1+\C_2)={\rm rank}(\C_g)$, where ${\rm rank}(X)$ denotes the rank of a matrix $X$. 
In this case, there must be $\bm p_c^T=(p_{1,c},\cdots, p_{n,c})\in \mathbb R^n$ such that
$(\C_1+\C_2)\bm p_c=\bm g_c$ (in many applications, $\C_1+\C_2$ is invertible and $\bm p_c=(\C_1+\C_2)^{-1}\bm g_c$).
Hence, we can decompose $\bm g=\bm g_0+\bm g_c$ where $\bm g_0=\bm g-\frac{1}{|\Omega|}\int_{\Omega} \bm g dx$,
$\bm g_c=\frac{1}{|\Omega|}\int_{\Omega}\bm gdx$, and thus $\int_{\Omega}\bm g_0 dx=\bm 0$.
Then the solution $(\boldsymbol u; \boldsymbol v; \bm p)$ can be decomposed according to 
$(\boldsymbol u; \boldsymbol v; \bm p)=(\boldsymbol u; \boldsymbol v; \bm p_0)+(\boldsymbol 0; \boldsymbol 0; \bm p_c)$ where
$\bm p_0^T=(p_{1,0},\cdots, p_{n,0})\in L^2_0(\Omega)\times\cdots \times L^2_0(\Omega) $ and $\bm p_c$ is a basic solution
of $(\C_1+\C_2)\bm p_c=\bm g_c$. Therefore we only need to consider the case when $\int_{\Omega}\bm g dx=\bm 0$.
\end{remark}
\noindent
Now we introduce the SPD matrix
\begin{equation}\label{eq:split_lambda}
\C  = {\sum_{i=1}^{4}\C _i}.
%R_v=
%\begin{bmatrix}
%R^{-1}_1&0&\dots &0\\
%0&R^{-1}_2&\vdots&\vdots\\
%\vdots&\dots&\ddots&\vdots\\
%0&\dots&0& R^{-1}_n
%\end{bmatrix}, 
\end{equation}
 %\qquad \text{i.e} \quad \C , \C^{-1} 
As we will see, it will play an important role in the definition of proper norms and the splitting \eqref{eq:split_lambda} 
in our analysis. 
The crucial idea is that we equip the Hilbert spaces $\bU, \bV,\bP$ with parameter-matrix-dependent norms
$\|\cdot\|_{\bU}$, $\|\cdot\|_{\bV}$, $\|\cdot\|_{\bP}$ induced by the following {\it inner products}:
\begin{subequations}\label{norms}
\begin{align}
(\bu,\bw)_{\bU}&= (\ep(\bu),\ep(\bw)) + \lambda(\divv\bu,\divv\bw),\label{norms-u}\\
(\bv,\bz)_{\bV}&= \sum_{i=1}^n(R_i^{-1}\bv_i,\bz_i) + (\C^{-1}  \Divv\bv,\Divv\bz),\label{norms-v}\\
(\bp,\bq)_{\bP}&= (\C \bp,\bq),\label{norms-p}
\end{align}
\end{subequations}
where $\bp^T=(p_1,\dots,p_n)$, $\bv^T=(\bv_1^T,\dots,\bv_n^T)$, $(\Divv\bv)^T= (\divv\bv_1, \ldots, \divv\bv_n)$. 
%$\bp = \left[\begin{array}{c}p_1 \\ \ \vdots \\ p_n \end{array} \right],\bv = \left[\begin{array}{c}\bv_1 \\ \ \vdots \\ \bv_n \end{array} \right],
%\Divv\bv= \left[\begin{array}{c}\divv\bv_1 \\ \ \vdots \\ \divv\bv_n \end{array} \right]$. 

It is easy to show that \eqref{norms-u}-\eqref{norms-p} are indeed inner products on $\bU, \bV,\bP$ respectively. 
It should be noted that $\Divv\bv, \Divv\bz$ and $\bm p,\bm q$ are vectors and the SPD matrix $\C$ is used to
define the norms. These novel parameter-matrix-dependent norms play a key role in the analysis of the uniform
stability for the MPET model. 
We further point out that for $n=1$, the norms defined by~\eqref{norms} are slightly different, but equivalent
to the norms that were used in~\cite{HongKraus2017parameter} to establish the parameter-robust inf-sup stability
of the three-field formulation of Biot's model of consolidation. 

The main result of this section is a proof of the uniform well-posedness of problem~\eqref{eq:8} under the norms induced 
by~\eqref{norms}. Firstly, directly related to problem~\eqref{eq:8}, we introduce the bilinear form
  \begin{align*}
&\mathcal{A}((\bu;\bv; \bm p),(\bw;\bz;\bm q))= (\ep(\bu), \ep(\bw))+ \lambda  (\divv\bu,\divv \bw) -\sum_{i=1}^{n} (p_i,\divv \bw) 
+\sum_{i=1}^{n} (R^{-1}_i\bv_i,\bz_i) \\&- \sum_{i=1}^{n}(p_i,\divv \bz_i)
 -\sum_{i=1}^{n}(\divv\bu,q_i)-\sum_{i=1}^{n}(\divv\bv_i,q_i)  -\sum_{i=1}^{n}(\alpha_{p_i}+\alpha_{ii})(p_i,q_i) 
+ \sum_{i=1}^{n}\sum_{\substack{j=1\\j\neq i}}^{n} \alpha_{ji} (p_j,q_i),
 \end{align*}
which, in view of the definition of the matrices $\C_1$ and $\C_2$, can be written in the form
 \begin{align*}
 \mathcal{A}((\bu;\bv; \bm p),(\bw;\bz;\bm q))&= (\ep(\bu),\ep(\bw)) + \lambda (\divv\bu,\divv\bw) -(\sum_{i=1}^{n}p_i,\divv \bw)
  +\sum_{i=1}^n(R_i^{-1}\bv_i,\bz_i) - (\bp, \Divv \bz)
\\&-(\divv\bu,\sum_{i=1}^{n}q_i)- (\Divv \bv,\bq) - ((\C_1+\C_2)\bp,\bq).
  \end{align*}
  %%%%%%%%%%%%%%%%%%%
  %%%%%%%%%   Boundedness    Boundedness
  %%%%%%%%%%%%%%%%%%%%%%%%%%%%%%%%%%%%%%%%%%%%%
  
Then the following theorem shows the boundedness of $\mathcal{A}((\cdot;\cdot; \cdot),(\cdot;\cdot;\cdot))$ in the norms
induced by~\eqref{norms}.

\begin{theorem}\label{bound} 
There exists a constant $C_{b}$ independent of  the parameters $\lambda,R_i^{-1}, \alpha_{p_i}, {\alpha}_{ij}, i,j=1,\dots,n$ and the network scale $n$, such that  for any $(\bm u; \bm v; \bm p)\in \bm U\times\bm V\times \bm P, (\bm w;\bm z;\bm q)\in \bm U\times\bm V\times \bm P$
\begin{equation*}
|\mathcal A((\boldsymbol u; \boldsymbol v;\bm p), (\boldsymbol w;\boldsymbol z;\bm q))|\le C_{b} (\|\bm u\|_{\bm U}+\|\bm v\|_{\bm V}+\|\bm p\|_P)  (\|\bm w\|_{\bm U}+\|\bm z\|_{\bm V}+\|\bm q\|_P).
\end{equation*}
\end{theorem}
\begin{proof}
From the definition of the bilinear form, by using Cauchy's inequality, we obtain
 \begin{align*}
 \mathcal{A}((\bu;\bv;\bm p),(\bw;\bz;\bm q))&= (\ep(\bu),\ep(\bw)) + \lambda (\divv\bu,\divv\bw) -(\sum_{i=1}^{n}p_i,\divv \bw)
 \\& +\sum_{i=1}^n(R_i^{-1} \bv_i,\bz_i) - (\bp, \Divv \bz)
-(\divv\bu,\sum_{i=1}^{n}q_i)- (\Divv \bv,\bq) - ((\C_1+\C_2)\bp,\bq)\\
&\leq \|\ep(\bu)\|\|\ep(\bw)\| + \lambda \|\divv\bu\| \|\divv\bw\| + \frac{1}{\sqrt{\lambda_0}}\|\sum_{i=1}^{n}p_i\|\sqrt{\lambda_0} \|\divv \bw\|
\\&+\sum_{i=1}^n(R_i^{-1}\bv_i,\bv_i)^{\frac{1}{2}}(R_i^{-1}\bz_i,\bz_i)^{\frac{1}{2}} +\| \C^{\frac{1}{2}}\bp\|\| \C^{-\frac{1}{2}}\Divv\bz\|+\sqrt{\lambda_0} \|\divv\bu\|  \frac{1}{\sqrt{\lambda_0}}\|\sum_{i=1}^{n}q_i\|\\&+
\|\C^{-\frac{1}{2}}\Divv\bv\| \|\C^{\frac{1}{2}}\bq\| +\|(\C_1+\C_2)^{\frac{1}{2}}\bp\| \|(\C_1+\C_2)^{\frac{1}{2}}\bq\| .
 \end{align*}
Then, another application of Cauchy's inequality, in view of the definition of $\C_4$, yields
 \begin{align*}
 \mathcal{A}((\bu;\bv;\bm p),(\bw;\bz;\bm q))&\leq \|\ep(\bu)\|\|\ep(\bw)\| + \lambda \|\divv\bu\| \|\divv\bw\| + \|\C_4^{\frac{1}{2}}\bp\|\sqrt{\lambda_0} \|\divv \bw\|\\
 &+\Big(\sum_{i=1}^n(R_i^{-1}\bv_i,\bv_i)\Big)^{\frac{1}{2}}\Big(\sum_{i=1}^n(R_i^{-1}\bz_i,\bz_i)\Big)^{\frac{1}{2}}+\| \C^{\frac{1}{2}}\bp\|\| \C^{-\frac{1}{2}}\Divv\bz\|\\&+\sqrt{\lambda_0} \|\divv\bu\|\|\C_4^{\frac{1}{2}}\bq\|+
 \|\C^{-\frac{1}{2}}\Divv\bv\| \|\C^{\frac{1}{2}}\bq\| +\|(\C_1+\C_2)^{\frac{1}{2}}\bp\|\|(\C_1+\C_2)^{\frac{1}{2}}\bq\|.
 \end{align*}
 \end{proof}
 %%%%%%%%%%%%%%%%%%
 %%%%%%%%%%%%%%%%%%%%%
 
   %%%%%%%%%%%%%%%%%%%
    %%%%%%%%%   Coercivity    Coercivity
  %%%%%%%%%%%%%%%%%%%%%%%%%%%%%%%%%%%%%%%%%%%%%
 
Before we study the uniform inf-sup condition for the MPET equations, we recall the following well known results, see, e.g.~\cite{ Brezzi1974existence,Boffi2013mixed}:
\begin{lemma}\label{divinf-sup}
There exists a constant $\beta_v > 0$ such that
\begin{align}
\inf_{q\in P_i} \sup_{\bv\in \bV_i} \frac{({\rm div} \bv, q)}{\|\bv\|_{\rm \div}\|q\|} \geq \beta_v,~i=1,\dots,n.
\end{align}
\end{lemma}

\begin{lemma}\label{stokesinf-sup}
There exists a constant $\beta_s > 0$ such that
\begin{align}\displaystyle
\inf_{(q_1,\cdots,q_n)\in P_1\times\cdots\times P_n} \sup_{\bu\in \bU} \frac{({\rm div} \bu, \sum\limits_{i=1}^n q_i)}{\|\bu\|_1\|\sum\limits_{i=1}^n q_i\|} \geq \beta_s.  
\end{align}
\end{lemma}

\noindent
Furthermore, we summarize some useful properties of the matrix $\C$ in the following lemma.
\begin{lemma} \label{etauu}
Let $\tilde{\C}= \C_3 + \C_4, \tilde{\C}^{-1}=(\tilde b_{ij})_{n\times n}$, then $\tilde{\C}$ is SPD and for any n-dimensional vector $\bm x$, we have
\begin{align}
&(\C {\boldsymbol x},{\boldsymbol x}) \geq (\tilde{\C }{\boldsymbol x},{\boldsymbol x})\geq (\C_3 {\boldsymbol x},{\boldsymbol x}),\label{etauu:1}
\\&(\C^{-1} {\boldsymbol x},{\boldsymbol x}) \leq (\tilde{\C}^{-1}{\boldsymbol x},{\boldsymbol x})\leq (\C_3^{-1}{\boldsymbol x},{\boldsymbol x})=R^{-1}(\bm x, \bm x).\label{etauu:2}
\end{align}
Also,
{
\begin{align}
&0<\sum_{\substack{i=1}}^n\sum_{\substack{j=1}}^n \tilde b_{ij}\le \lambda_0.\label{etauu:3}
\end{align}
}
\end{lemma}
\begin{proof}

From the definitions of $\C_3, \C_4$, noting that $\C_3$ is SPD and $\C_4$ is SPSD, it is obvious that $\tilde {\C}$ is SPD. 

From the definition of $\C$, noting that $\C_1$ and $\C_2$ are SPSD, we infer the estimates
\begin{align*}
&(\C {\boldsymbol x},{\boldsymbol x}) \geq (\tilde{\C }{\boldsymbol x},{\boldsymbol x})\geq (\C_3 {\boldsymbol x},{\boldsymbol x}), ~~(\C^{-1} {\boldsymbol x},{\boldsymbol x}) \leq (\tilde{\C}^{-1}{\boldsymbol x},{\boldsymbol x})\leq (\C_3^{-1}{\boldsymbol x},{\boldsymbol x})=R^{-1}(\bm x, \bm x).
\end{align*}
Next, we show that
$$
\sum_{\substack{i=1}}^n\sum_{\substack{j=1}}^n \tilde b_{ij}\le { \lambda_0}.
$$
From the definitions of $\C_3, \C_4$ and $\tilde {\C}$, we have 
\begin{align*}
\tilde{\C}=
\begin{bmatrix}
R+\frac{1}{\lambda_0} & \frac{1}{\lambda_0} & \dots &\frac{1}{\lambda_0}\\
\frac{1}{\lambda_0}   &\ddots             &\ddots& \vdots\\
\vdots&\ddots& \ddots & \frac{1}{\lambda}     \\
\frac{1}{\lambda_0} & \dots & \frac{1}{\lambda_0}  &R+\frac{1}{\lambda_0}
\end{bmatrix}.
\end{align*}
Now, using the Sherman-Morrison-Woodbury formula, we find
\begin{align*}
\tilde{\C}^{-1}=(\C_3 - \tilde{\bm \lambda} \bm e^T)^{-1} =
\C_3 ^{-1} + \frac{\C_3^{-1} \tilde{\bm \lambda} \bm e^T \C_3^{-1}}{1-\bm e^T \C_3^{-1}\tilde{\bm \lambda}}
\intertext{where}
\tilde{\bm \lambda}=(\underbrace{\frac{1}{\lambda_0},\dots,\frac{1}{\lambda_0}}_n)^T,\qquad \bm e=(\underbrace{-1,\dots,-1}_n)^T.
\end{align*}
Further, noting that
\begin{align*}
\C_3^{-1}&=
\begin{bmatrix}
\frac{1}{R} & 0 & \dots &0\\
0   &\ddots             &\ddots& \vdots\\
\vdots&\ddots& \ddots & 0     \\
0 & \dots & 0  &\frac{1}{R}
\end{bmatrix}
=\frac{1}{R}I_{n\times n},
\end{align*}
where $I_{n\times n}$ is the $n$-th order identity matrix, we obtain 
\begin{align*}
\C_3^{-1} \tilde{\bm \lambda} \bm e^T \C_3^{-1}&=
\left(\frac{1}{R}I_{n\times n}\right)
%\cdot
\begin{bmatrix}
-\frac{1}{\lambda_0} & \dots & \dots &-\frac{1}{\lambda_0}\\
\vdots   &\ddots             & & \vdots\\
\vdots& & \ddots & \vdots     \\
-\frac{1}{\lambda_0} &\dots & \dots  &-\frac{1}{\lambda_0}
\end{bmatrix}
%\cdot
\left(\frac{1}{R}I_{n\times n}\right)=
\begin{bmatrix}
-\frac{1}{R^2\lambda_0} & \dots & \dots &-\frac{1}{R^2\lambda_0}\\
\vdots   &\ddots             & & \vdots\\
\vdots& & \ddots & \vdots     \\
-\frac{1}{R^2\lambda_0} &\dots & \dots  &-\frac{1}{R^2\lambda_0}
\end{bmatrix}
\end{align*}
and 
\begin{align*}
\bm e^T \C_3^{-1} \tilde{\bm \lambda}&=
(-1,\dots,\dots,-1)
\begin{bmatrix}
\frac{1}{R} & 0 & \dots &0\\
0   &\ddots             &\ddots& \vdots\\
\vdots&\ddots& \ddots & 0     \\
0 & \dots & 0  &\frac{1}{R}
\end{bmatrix}
\begin{bmatrix}
\frac{1}{\lambda_0}\\
\vdots\\
\vdots\\
\frac{1}{\lambda_0}
\end{bmatrix}
=\sum_{i=1}^{n}\frac{-1}{R\lambda_0}=-\frac{n}{R\lambda_0}
\end{align*}
which implies that
$$\frac{1}{1-\bm e^T \C_3^{-1} \tilde{\bm \lambda}} = \frac{R\lambda_0}{R\lambda_0+n}.$$
Now we can calculate $\tilde{\C}^{-1}$ as follows:
\begin{align*}
\tilde{\C}^{-1}&=\C_3 ^{-1} + \frac{\C_3^{-1} \tilde{\bm \lambda} \bm e^T \C_3^{-1}}{1-\bm e^T \C_3^{-1} \tilde{\bm \lambda}} = 
\begin{bmatrix}
\frac{1}{R} & 0 & \dots &0\\
0   &\ddots             &\ddots& \vdots\\
\vdots&\ddots& \ddots & 0     \\
0 & \dots & 0  &\frac{1}{R}
\end{bmatrix}
+\frac{R\lambda_0}{R\lambda_0+n}
\begin{bmatrix}
-\frac{1}{R^2\lambda_0} & \dots & \dots &-\frac{1}{R^2\lambda_0}\\
\vdots   &\ddots             & & \vdots\\
\vdots& & \ddots & \vdots     \\
-\frac{1}{R^2\lambda_0} &\dots & \dots  &-\frac{1}{R^2\lambda_0}
\end{bmatrix}
\\&=
\begin{bmatrix}
\frac{1}{R}-\frac{1}{R(R\lambda_0+n)}& -\frac{1}{R(R\lambda_0+n)} & \dots &-\frac{1}{R(R\lambda_0+n)}\\
-\frac{1}{R(R\lambda_0+n)}   &\frac{1}{R}-\frac{1}{R(R\lambda_0+n)}             &\dots& -\frac{1}{R(R\lambda_0+n)}\\
\vdots&\vdots& \ddots &\vdots   \\
-\frac{1}{R(R\lambda_0+n)} &-\frac{1}{R(R\lambda_0+n)} &\dots&\frac{1}{R}-\frac{1}{R(R\lambda_0+n)}
\end{bmatrix}
=(\tilde b_{ij})_{n\times n}.
\end{align*}
Finally, we conclude
\begin{align*}
\sum_{i=1}^{n}\sum_{j=1}^{n}\tilde{b}_{ij}&=\frac{n}{R}-\frac{n^2}{R(R\lambda_0+n)} = \frac{nR\lambda_0+n^2-n^2}{R(R\lambda_0+n)}= \frac{n\lambda_0}{(R\lambda_0+n)}
\leq \frac{n\lambda_0}{n}
=\lambda_0.
\end{align*}
\end{proof}

We are ready to prove the uniform inf-sup condition for $\mathcal{A}((\cdot;\cdot; \cdot),(\cdot;\cdot;\cdot))$
in the norms induced by~\eqref{norms}.
%$\mathcal{A}((\bu;\bv;\bp),(\bw;\bz;\bq))$ is as follows:
\begin{theorem}\label{stability:continuous}
There exists a constant $\omega> 0$ independent of the parameters $\lambda,R_i^{-1}, \alpha_{p_i}, {\alpha}_{ij}, i,j=1,\dots,n$ and the network scale $n$, such that
\begin{align*}
\inf_{(\bu;\bv;\bp)\in \bU\times \bV\times \bP}\sup_{(\bw;\bz;\bq)\in  \bU\times \bV\times \bP}\frac{\mathcal{A}((\bu;\bv;\bp),(\bw;\bz;\bq))}{( \|\bu\|_{\bU}+\|\bv\|_{\bV}+ \|\bp\|_{\bP})( \|\bw\|_{\bU}+ \|\bz\|_{\bV}+ \|\bq\|_{\bP})} \geq \omega.
\end{align*}
\end{theorem}
\begin{proof}
For any $(\bu;\bv; \bm p)=(\bu;\bv_1,\dots,\bv_n;p_1,\dots,p_n)\in \bU\times \bV_1\times\dots\times \bV_n\times P_1\times\dots\times P_n$, by Lemma \ref{divinf-sup}, there exist
\begin{align}
\bm \psi_i &\in \bm V_i ~~\hbox{such that}~~\divv\bm \psi_i = \sqrt{R}p_i ~~\hbox{and}~~ \|\bm \psi_i\|_{\divn} \leq { \beta_{d}^{-1}\sqrt{R}\|p_i\|} \label{eq:v3},~~ i=1,\dots,n;
\end{align}
%$\bm \psi=(\bv_{n+1},\dots,\bv_{2n})^T$\\\\
%For any $(\bu,\bv_1,\dots,\bv_n,p_1,\dots,p_n)\in \bU\times \bV_1\times\dots\times \bV_n\times P_1\times\dots\times P_n$,
and by Lemma \ref{stokesinf-sup}, there exists
\begin{align}\label{eq:u3}
\bu_0 &\in \bU~~\hbox{such that}~~ 
\divv\bu_0 =\frac{1}{\sqrt{\lambda_0}} (\sum_{i=1}^{n}p_i),~~\|\bu_0\|_1 \leq \beta_{s}^{-1}\frac{1}{\sqrt{\lambda_0}}\|\sum_{i=1}^{n}p_i\|.
\end{align}
Choose
\begin{align}\label{choice}
\bw = \delta \bu - \frac{1}{\sqrt{\lambda_0}}\bu_0,~~
\bz_i= \delta \bv_i  -\sqrt{R}\bm \psi_i,~i=1,\dots,n,~~
\bq = -\delta \bp - \C^{-1} \Divv \bv,
\end{align}
where $\delta$ is a positive constant to be determined later.\\

Now let us verify the boundedness of $(\bw;\bz;\bm q)$ by $(\bu;\bv;\bm p)$ in the combined norm.
Let $\bm \psi^T=(\bm \psi_1^T, \dots, \bm \psi_n^T)$, then $\bm z= \delta \bv  -\sqrt{R}\bm \psi$.

Firstly, by \eqref{eq:u3}, we have 
\begin{align*}
(\frac{1}{\sqrt{\lambda_0}}\bu_0, \frac{1}{\sqrt{\lambda_0}}\bu_0)_{\bU} &= (\ep(\frac{1}{\sqrt{\lambda_0}}\bu_0),\ep(\frac{1}{\sqrt{\lambda_0}}\bu_0)) +  \lambda(\divv(\frac{1}{\sqrt{\lambda_0}}\bu_0),\divv(\frac{1}{\sqrt{\lambda_0}}\bu_0)) 
\\&\leq \frac{1}{\lambda_0}(\ep(\bu_0),\ep(\bu_0)) +(\divv\bu_0 ,\divv\bu_0)
\leq \frac{1}{\lambda_0}(\ep(\bu_0),\ep(\bu_0)) +\frac{1}{\lambda_0}(\sum_{i=1}^{n}p_i,\sum_{i=1}^{n}p_i) 
\\&\leq \frac{1}{\lambda_0}\beta_{s}^{-2}\frac{1}{\lambda_0}\|\sum_{i=1}^{n}p_i\|^2+ \frac{1}{\lambda_0}\|\sum_{i=1}^{n}p_i\|^2
\leq \frac{1}{\lambda_0}(\beta_{s}^{-2}\frac{1}{\lambda_0}+ 1)\|\sum_{i=1}^{n}p_i\|^2 
\leq \frac{1}{\lambda_0}(\beta_{s}^{-2}+ 1)\|\sum_{i=1}^{n}p_i\|^2 
\\&\leq \frac{1}{\lambda_0}(\beta_{s}^{-2}+ 1)\|\sum_{i=1}^{n}p_i\|^2=(\beta_{s}^{-2}+ 1)(\C_4\bp,\bp) \leq(\beta_{s}^{-2}+ 1)\|\bp\|^2_{\bP},
\end{align*}
which implies that
\begin{equation}\label{bound:u}
\|\bw\|_{\bU} \leq  \delta \|\bu\|_{\bU} +\sqrt{(\beta_{s}^{-2}+ 1)}\|\bp\|_{\bP}.
\end{equation}
Secondly, by \eqref{etauu:2} and \eqref{eq:v3}, we have
\begin{align*}
(\sqrt{R}\bm \psi,\sqrt{R}\bm \psi)_{\bV} &= \sum_{i=1}^n(R_i^{-1}\sqrt{R}\bm \psi_i,\sqrt{R}\bm \psi_i)+(\C^{-1} \Divv(\sqrt{R}\bm \psi),\Divv(\sqrt{R}\bm \psi))
\\&\leq R\sum_{i=1}^n(R_i^{-1}\bm \psi_i,\bm \psi_i) + R^{-1} (\Divv(\sqrt{R}\bm \psi),\Divv(\sqrt{R}\bm \psi))
\leq \sum_{i=1}^n(\bm \psi_i,\bm \psi_i) +(\Divv\bm \psi,\Divv\bm \psi) 
\\&= \sum_{i=1}^{n}\|\bm \psi_i\|^2+ \sum_{i=1}^{n}(\divv\bm \psi_i,\divv\bm \psi_i)
= \sum_{i=1}^{n} \|\bm \psi_i\|^2_{\divn}\leq \sum_{i=1}^{n} \beta_{d}^{-2}R\|p_i\|^2 = \beta_{d}^{-2}R\|\bp\|^2 \leq\beta_{d}^{-2}\|\bp\|^2_{\bP},
\end{align*}
which implies that
\begin{equation}\label{bound:z}
\|\bz\|_{\bV} \leq \delta \|\bv\|_{\bV}+ \beta^{-1}_{d}\|\bp\|_{\bP}.
\end{equation}
Thirdly, there holds
\begin{equation}\label{bound:q}
\|\bq\|_{\bP} \leq \delta \|\bp\|_{\bP} + \|\bv\|_{\bV}
\end{equation}
since $(\C^{-1} \Divv\bv,\C^{-1} \Divv\bv)_{\bP}  =(\Divv\bv,\C^{-1} \Divv\bv) \leq (\bv,\bv)_{\bV}$.

Collecting the estimates \eqref{bound:u}, \eqref{bound:z} and \eqref{bound:q}, we obtain
$$
\|\bw\|_{\bU} + \|\bz\|_{\bV} + \|\bq\|_{\bP} \leq (\delta +1+ { \beta^{-1}_{d}}+
\beta_{s}^{-1} )\big( \|\bu\|_{\bU} + \|\bv\|_{\bV} + \|\bp\|_{\bP}\big)
$$
and hence the desired boundedness estimate.\\

Next, we show the coercivity of $\mathcal{A}((\bu; \bv;\bm p),(\bw; \bz; \bm q))$. Using the definition of $\mathcal{A}((\bu; \bv;\bm p),(\bw; \bz; \bm q))$ and that of $(\bw; \bz; \bm q)$ from \eqref{choice}, we find
\begin{align*}
  & \mathcal A((\boldsymbol u; \boldsymbol v; \bm p), (\boldsymbol w;\boldsymbol z;\bm q))
   \\=& (\ep(\bu),\ep(\bw)) + \lambda (\divv\bu,\divv\bw)-(\sum_{i=1}^{n}p_i,\divv \bw)
 \\&+\sum_{i=1}^n(R_i^{-1} \bv_i,\bz_i) - (\bp,\Divv \bz)-(\divv\bu,\sum_{i=1}^{n} q_i)
 - (\Divv \bv, \bq)- ((\C_1+\C_2)\bp,\bq)
 \\=&(\ep(\bu),\ep(\delta \bu - \frac{1}{\sqrt{\lambda_0}}\bu_0)) + \lambda (\divv\bu,\divv(\delta \bu - \frac{1}{\sqrt{\lambda_0}}\bu_0)) 
 -( \sum_{i=1}^{n}p_i, \divv (\delta \bu - \frac{1}{\sqrt{\lambda_0}}\bu_0))
\\&+ \sum_{i=1}^n(R_i^{-1}\bv_i,(\delta \bv_i-\sqrt{R}\bm \psi_i))-(\Divv(\delta \bv  - \sqrt{R}\bm \psi),\bp) 
-((\underbrace{\divv\bu,\dots,\divv\bu}_n)^T,-\delta \bp - \C^{-1}  \Divv \bv) 
\\ &-(\Divv\bv,-\delta \bp - \C^{-1}  \Divv \bv) - ((\C_1+\C_2)\bp,(-\delta \bp - \C^{-1} \Divv \bv)).
\end{align*}
Using \eqref{eq:v3} and \eqref{eq:u3}, we therefore get
\begin{align*}
& \mathcal A((\boldsymbol u; \boldsymbol v; \bm p), (\boldsymbol w;\boldsymbol z;\bm q))\\
=&\delta(\ep(\bu),\ep(\bu)) - \frac{1}{\sqrt{\lambda_0}}(\ep(\bu),\ep(\bu_0)) + \delta\lambda (\divv\bu,\divv \bu)-\frac{\lambda}{\sqrt{\lambda_0}} (\divv\bu,\divv\bu_0)
 - \delta(\sum_{i=1}^{n}p_i, \divv \bu) \\&+\frac{1}{\sqrt{\lambda_0}}(\sum_{i=1}^{n}p_i, \divv\bu_0)
+ \delta\sum_{i=1}^n(R_i^{-1}\bv_i,\bv_i)  - \sqrt{R}\sum_{i=1}^n(R_i^{-1}\bv_i,\bm \psi_i)-\delta(\Divv\bv,\bp)  + \sqrt{R}(\Divv\bm \psi,\bp) 
\\ &+\delta((\underbrace{\divv\bu,\dots,\divv\bu}_n)^T, \bp) + (\C^{-1} (\underbrace{\divv\bu,\dots,\divv\bu}_n)^T,\Divv \bv) + \delta( \bp, \Divv\bv) \\&+ ( \C^{-1}  \Divv \bv , \Divv\bv) + \delta((\C_1+\C_2)\bp,\bp) 
+ ((\C_1+\C_2)\C^{-1} \bp, \Divv \bv)
\\=&\delta(\ep(\bu),\ep(\bu)) - \frac{1}{\sqrt{\lambda_0}}(\ep(\bu),\ep(\bu_0))  + \delta\lambda (\divv\bu,\divv \bu)-\frac{\lambda}{\lambda_0}(\divv\bu,\sum_{i=1}^{n}p_i) 
 +\frac{1}{\lambda_0}(\sum_{i=1}^{n}p_i,\sum_{i=1}^{n}p_i)
\\&+ \delta\sum_{i=1}^n(R_i^{-1} \bv_i,\bv_i)  -  \sqrt{R}\sum_{i=1}^n(R_i^{-1} \bv_i,\bm \psi_i) + R\sum_{i=1}^{n}(p_i,p_i) 
 + (\C^{-1} ((\underbrace{\divv\bu,\dots,\divv\bu}_n)^T,\Divv \bv)\\&+ (\C^{-1} \Divv \bv , \Divv\bv) 
 + \delta((\C_1+\C_2)\bp,\bp)+ ((\C_1+\C_2)\C^{-1} \bp, \Divv \bv) .
 \end{align*}
Using Young's inequality, it follows that  
\begin{align}\label{Young}
& \nonumber\mathcal A((\boldsymbol u; \boldsymbol v; \bm p), (\boldsymbol w;\boldsymbol z;\bm q))\\
&\nonumber\geq\delta(\ep(\bu),\ep(\bu)) -\frac{1}{2}\frac{1}{\sqrt{\lambda_0}}\epsilon_1(\ep(\bu),\ep(\bu)) - \frac{1}{2}\frac{1}{\sqrt{\lambda_0}}\epsilon^{-1}_1(\ep(\bu_0),\ep(\bu_0))+ \delta\lambda (\divv\bu,\divv \bu)- \lambda(\divv\bu,\divv\bu) 
\\&\nonumber-\frac{\lambda}{4\lambda^2_0}(\sum_{i=1}^{n}p_i,\sum_{i=1}^{n}p_i)+\frac{1}{\lambda_0}(\sum_{i=1}^{n}p_i, \sum_{i=1}^{n}p_i)
+ \delta\sum_{i=1}^n(R^{-1}_i \bv_i,\bv_i) - \frac{1}{2}\epsilon_2\sum_{i=1}^n(R^{-1}_i \bv_i,\bv_i)-\frac{1}{2}\epsilon_2^{-1}R\sum_{i=1}^n(R^{-1}_i \bm \psi_i,\bm \psi_i)
\\&\nonumber + R\sum_{i=1}^{n}(p_i,p_i)
- (\C^{-1} (\underbrace{\divv\bu,\dots,\divv\bu}_n)^T,(\underbrace{\divv\bu,\dots,\divv\bu}_n)^T)- \frac{1}{4}(\C^{-1} \Divv \bv,\Divv \bv) 
\\&+ (\C^{-1} \Divv \bv , \Divv\bv) 
+ \delta((\C_1+\C_2)\bp,\bp)-\frac{1}{4}((\C_1+\C_2)\C^{-1} \Divv \bv,\C^{-1} \Divv \bv)-((\C_1+\C_2)\bp,\bp).
\end{align}
From the definition of $\C$ and noting that both $\C_3$ and $\C_4$ are SPSD, we conclude
\begin{align}\label{eta:SSPD}
&\nonumber(\C^{-1} \Divv \bv,\Divv \bv)-((\C_1+\C_2)\C^{-1} \Divv \bv,\C^{-1}\Divv \bv)
\\&\nonumber=(\C^{-1}  \Divv \bv,\C \C^{-1} \Divv \bv)-(\C^{-1}  \Divv \bv,(\C_1+\C_2)\C^{-1}  \Divv \bv)
\\&=(\C^{-1}  \Divv \bv,(\C_3+\C_4)\C^{-1}  \Divv \bv) \geq 0.
\end{align}
Furthermore, by \eqref{etauu:3} from Lemma~\ref{etauu}, we have that
\begin{align}\label{lemma:etauu}
(\C^{-1} (\underbrace{\divv\bu,\dots,\divv\bu}_n)^T,(\underbrace{\divv\bu,\dots,\divv\bu}_n)^T)&=\big(\sum_{i=1}^{n}\sum_{j=1}^{n}\tilde{b}_{ij}\big)(\divv\bu,\divv\bu)\leq \lambda_0 (\divv\bu,\divv\bu).
\end{align}
Collecting \eqref{Young}, \eqref{eta:SSPD}, \eqref{lemma:etauu}, the estimates from \eqref{eq:v3} and \eqref{eq:u3}, and noting that $\lambda_0=\max\{\lambda,1\}$, the proof continues as follows:
\begin{align*}
&\mathcal A((\boldsymbol u; \boldsymbol v; \bm p), (\boldsymbol w;\boldsymbol z;\bm q))\\
&\geq(\delta- \frac{1}{2}\frac{1}{\sqrt{\lambda_0}}\epsilon_1)(\ep(\bu),\ep(\bu))
- \frac{1}{2}\frac{1}{\sqrt{\lambda_0}}\epsilon^{-1}_1 \beta^{-2}_{s}\frac{1}{\lambda_0}( \sum_{i=1}^{n}p_i,
\sum_{i=1}^{n}p_i)+ (\delta-1) \lambda(\divv\bu,\divv\bu) \\
&+\frac{3}{4\lambda_0}( \sum_{i=1}^{n}p_i,\sum_{i=1}^{n}p_i)
+ (\delta- \frac{1}{2}\epsilon_2)\sum_{i=1}^n(R_i^{-1}\bv_i,\bv_i)-\frac{1}{2}\epsilon_2^{-1}\sum_{i=1}^n(\bm \psi_i,\bm \psi_i)+ R\sum_{i=1}^{n}(p_i,p_i)\\
& -  (\lambda_0-\lambda+\lambda)(\divv\bu,\divv\bu)+ \frac{1}{2}( \C^{-1}  \Divv \bv , \Divv\bv) + (\delta-1)((\C_1+\C_2)\bp,\bp).
\end{align*}
Now, let $\epsilon_1 := 2\beta_{s}^{-2} , \epsilon_2 := 2{ \beta_{d}^{-2}}$, and note that $\lambda_0=\max\{\lambda,1\}$
and $({\rm\div} \bm u,{\rm\div} \bm u)\leq (\ep(\bu),\ep(\bu))$. Then we obtain
\begin{align*}
&\mathcal A((\boldsymbol u; \boldsymbol v; \bm p), (\boldsymbol w;\boldsymbol z;\bm q))\\
&\geq(\delta- \beta_{s}^{-2}-1)(\ep(\bu),\ep(\bu)) - \frac{1}{4\lambda_0}( \sum_{i=1}^{n}p_i, \sum_{i=1}^{n}p_i)+ (\delta-2) \lambda(\divv\bu,\divv\bu) 
+\frac{3}{4\lambda_0}( \sum_{i=1}^{n}p_i,\sum_{i=1}^{n}p_i)
\\&+ (\delta-\beta_{d}^{-2})\sum_{i=1}^n(R_i^{-1}\bv_i,\bv_i)-\frac{1}{4} R\sum_{i=1}^{n}(p_i,p_i)
+ R\sum_{i=1}^{n}(p_i,p_i)+ \frac{1}{2}( \C^{-1}  \Divv \bv , \Divv\bv) + (\delta-1)((\C_1+\C_2)\bp,\bp),
\end{align*}
or, equivalently,
\begin{align*}
&\mathcal A((\boldsymbol u; \boldsymbol v; \bm p), (\boldsymbol w;\boldsymbol z;\bm q))\\
&\ge (\delta - \beta_{s}^{-2}-1) (\ep(\bu),\ep(\bu)) + (\delta-2)\lambda (\divv\bu,\divv \bu) 
+\frac{1}{2}(\C_4\bp,\bp) \\
&+  (\delta-\beta_{d}^{-2})\sum_{i=1}^n(R_i^{-1}\bv_i,\bv_i)+ \frac{3}{4}(\C_3\bp,\bp)
+ \frac{1}{2}(\C^{-1} \Divv \bv , \Divv\bv) + (\delta-1)((\C_1+\C_2)\bp,\bp).
\end{align*}
Finally, let
$\delta := \max\left\lbrace \beta_{s}^{-2} + \frac{1}{2} +1,\beta_{d}^{-2}+ \frac{1}{2},2+ \frac{1}{2} \right\rbrace$.
Then, using the definition of $\C$, we get the desired coercivity estimate
\begin{align*}
&\mathcal A((\boldsymbol u; \boldsymbol v; \bm p), (\boldsymbol w;\boldsymbol z;\bm q))\\
&=(\delta - \beta_{s}^{-2}-1) (\ep(\bu),\ep(\bu)) + (\delta-2)\lambda (\divv\bu,\divv \bu)
+(\delta-{ \beta_{d}^{-2}})\sum_{i=1}^n(R_i^{-1}\bv_i,\bv_i) \\
&+ \frac{1}{2}(\C^{-1} \Divv \bv , \Divv\bv) + \big(((\delta-1)(\C_1+\C_2)+\frac{3}{4}\C_3+\frac{1}{2}\C_4)\bp,\bp\big)\\
&\geq \frac{1}{2}\big(\|\bu\|^2_{\bU}+\|\bv\|^2_{\bV}+\|\bp\|^2_{\bP} \big).
\end{align*}
\end{proof}
The above theorem implies the following stability result.
\begin{corollary}\label{eq:67}
Let $(\boldsymbol u; \boldsymbol v;\bp)\in \boldsymbol U\times \boldsymbol V\times P$ be the solution
of~\eqref{eq:8}. Then there holds the estimate
\begin{equation}
\|\boldsymbol u\|_{\boldsymbol U}+\|\boldsymbol  v\|_{\boldsymbol V}+\|p\|_P\leq C_1 (\|\boldsymbol f\|_{\boldsymbol U^*}
+\|\bm g\|_{\bm P^*}),
\end{equation} 
for some positive constant $C_1$ that is independent of the parameters
$\lambda,R_i^{-1}, \alpha_{p_i}, {\alpha}_{ij}, i,j=1,\dots,n$ and the network scale $n$, 
where
$$\|\boldsymbol f\|_{\boldsymbol U^*}=
\sup\limits_{\boldsymbol w\in \boldsymbol U}\frac{(\boldsymbol f, \boldsymbol w)}{\|\boldsymbol w\|_{\boldsymbol U}},
\quad \|\bm g\|_{\bm P^*}=
\sup\limits_{\bm q\in \bm P}\frac{(\bm g,\bm q)}{\|\bm q\|_{\bm P}}=\|\C^{-\frac{1}{2}} \bm g\|, \quad \bm g^T=(g_1,\cdots,g_n).$$
\end{corollary}
\begin{remark}
We want to emphasize that the parameter ranges as specified in~\eqref{parameter:range}
are indeed relevant since the variations of the model parameters are quite large in many applications.
For that reason, Theorem~\ref{bound} and Theorem~\ref{stability:continuous} are very important fundamental results that provide the
parameter-robust stability of the model~\eqref{eq:8a}--\eqref{eq:8c}. We also point out that the matrix technique plays an interesting
role for proving the uniform stability. 
\end{remark}
{
\begin{remark}
Let $\C=(\gamma_{ij})_{n\times n}, \C^{-1}=(\tilde{\gamma}_{ij})_{n\times n}$ and define
\begin{equation}\label{Preconditioner:B}
\mathcal{B}:=\left[\begin{array}{ccc}
\mathcal{B}_{\bm u}^{-1} & \bm 0 & \bm 0 \\
\bm 0 & \mathcal{B}_{\bm v}^{-1}& \bm 0 \\
\bm 0& \bm 0& \mathcal{B}_{\bm p}^{-1}  
\end{array}
\right],
\end{equation}
where 
$$
\mathcal{B}_{\bm u}=-\div \ep-\lambda \nabla \div,
$$
\begin{align*}
\mathcal{B}_{\bm v}&=  
\begin{bmatrix}
R_1^{-1}I&0  & \dots &0  \\
0 &R_2^{-1} I& \dots &0 \\
\vdots &\vdots  & \ddots & \vdots  \\
0& 0 & \dots &R_n^{-1}I
\end{bmatrix}
-
\begin{bmatrix}
\tilde{\gamma}_{11}\nabla{\rm div} &\tilde{\gamma}_{12}\nabla{\rm div} & \dots &\tilde{\gamma}_{1n}\nabla{\rm div}\\
\tilde{\gamma}_{21}\nabla{\rm div} &\tilde{\gamma}_{22}\nabla{\rm div} & \dots &\tilde{\gamma}_{2n}\nabla{\rm div}\\
\vdots &\vdots  & \ddots & \vdots  \\
\tilde{\gamma}_{n1}\nabla{\rm div} &\tilde{\gamma}_{n2}\nabla{\rm div} & \dots &\tilde{\gamma}_{nn}\nabla{\rm div}
\end{bmatrix},
\end{align*}
\begin{align*}
\mathcal{B}_{\bm p}&=
\begin{bmatrix}
\gamma_{11}I&\gamma_{12}I& \dots &\gamma_{1n}I\\
\gamma_{21}I&\gamma_{22}I& \dots &\gamma_{2n}I\\
\vdots&\vdots  &\ddots&\vdots\\
\gamma_{n1}I&\gamma_{n2}I& \dots &\gamma_{nn}I\\
\end{bmatrix}.
\end{align*}
Inferring from the theory presented in~\cite{Mardal2011preconditioning}, 
Theorems~\ref{bound} and~\ref{stability:continuous} imply that the operator~$\mathcal{B}$ defined in~\eqref{Preconditioner:B}
is a uniform norm-equivalent (canonical) block-diagonal preconditioner for the operator $\mathcal{A}$ in~\eqref{operator:A},
robust in all model and discretization parameters, i.e., $\kappa (\mathcal{B} \mathcal{A}) = \mathcal{O}(1)$.
\end{remark}
}

\section{Uniformly stable and strongly mass conservative discretizations}\label{sec:uni_stab_disc_model}

There are various discretizations that meet the requirements for the proof of full parameter-robust stability as presented
in this section. They include conforming as well as nonconforming methods. In general, if
$\bm U_h/(\sum\limits_{i=1}^nP_{i,h})$ is a Stokes-stable pair
and $\bm V_{i,h}/P_{i,h}$ satisfy the $H(\div)$ inf-sup condition for $i=1,\ldots,n$, see~\eqref{inf-sup}, then the norm
that we have proposed in Section~\ref{sec:par_rob_stab_model} allows for the proof of full parameter-robust stability
%using a proof that follows the same ideas as used in
using similar arguments as in the proof of Theorem \ref{Dis:Stability}. 
{To give a few examples, the triplets $\bm U_h/\bm V_{i,h}/\bm P_{i,h}=CR_{l}/ RT_{l-1}/P_{l-1} (l=1,2)$
together with stabilization \cite{hansbo2003discontinuous, fortin1983non} results in a parameter-robust stable discretization
of the MPET model if the norms are defined as in Section \ref{sec:par_rob_stab_model}.  The same is true for the conforming
discretizations based on the spaces $P_2/{RT}_0/ P_0$ (in 2D), $P_2^{stab}/RT_0/P_0$ (in 3D), or $P_2/RT_1/P_1$. However,
the above-mentioned finite element methods do not have the property of strong mass conservation in the sense of Proposition~\ref{prop:mass_cons}
although they result in parameter-robust inf-sup stability under the norms we proposed in Section \ref{sec:par_rob_stab_model}.}

In recent years, DG methods have been developed to solve various problems
\cite{arnold1982interior,brezzi2000discontinuous,arnold2002unified, cockburn2007note,hong2012discontinuous} 
and some unified analysis for finite element including DG methods has recently been presented
in~\cite{hong2017unified,hong2018uniform}.
In this section, motivated by the works \cite{schtzau2002mixed, hong2016robust, honguniformly},
we propose discretizations of the MPET model problem~\eqref{eq:8}. These discretizations preserve the divergence
condition (namely equation~\eqref{eq:6,3}) pointwise, which results in a strong conservation of mass, see
Proposition~\ref{prop:mass_cons}. Furthermore, they are also locking-free when the Lam\'e parameter~$\lambda$
tends to $\infty$. %Firstly, we introduce some notation.

\subsection{Preliminaries and notation}
{By $\mathcal{T}_h$ we denote} a shape-regular triangulation of mesh-size $h$ of
the domain $\Omega$ into triangles $\{K\}$. We further denote by
$\mathcal{E}_h^{I}$ the set of all interior edges (or faces) of  $\mathcal{T}_h$ and by
$\mathcal{E}_h^{B}$ the set of all boundary edges (or faces); we set
$\mathcal{E}_h=\mathcal{E}_h^{I}\cup \mathcal{E}_h^{B}$.

For $s\geq 1$, we define
$$
H^s(\mathcal{T}_h)=\{\phi\in L^2(\Omega), \mbox{ such that } \phi|_K\in H^s(K) \mbox{ for all } K\in \mathcal{T}_h \}.
$$

%{The vector functions are represented in column-wise.}

As we consider discontinuous Galerkin (DG) discretizations, we also define some trace operators.
Let $e = \partial K_1 \cap \partial K_2$ be the common boundary (interface) of two subdomains $K_1$
and $K_2$ in $\mathcal{T}_h$ , and $\bm n_1$ and $\bm n_2$ be unit normal vectors to $e$ pointing to the
exterior of $K_1$ and $K_2$, respectively. For any edge (or face) $e
\in \mathcal{E}_h^{I}$ and a scalar $q\in H^1(\mathcal{T}_h)$, vector $\bm v \in H^1(\mathcal{T}_h)^d$ and tensor
$\bm \tau \in H^1(\mathcal{T}_h)^{d\times d}$, we define the
averages
\begin{equation*}
\begin{split}
\{\bm v\} &=\frac{1}{2}(\bm v|_{\partial K_1\cap e}\cdot \bm n_1-\bm
v|_{\partial K_2\cap e}\cdot \bm n_2), \quad 
\{\bm \tau\}=\frac{1}{2}(\bm \tau|_{\partial K_1\cap
e} \bm n_1-\bm \tau|_{\partial K_2\cap e} \bm n_2),
\end{split}
\end{equation*}
and jumps
\begin{equation*}
[q]=q|_{\partial K_1\cap e}-q|_{\partial K_2\cap e},\quad
[\bm v]=\bm v|_{\partial K_1\cap e}-\bm v|_{\partial K_2\cap e},\quad
 \Lbrack\bm v\Rbrack=\bm v|_{\partial K_1\cap e}\odot \bm n_1+\bm v|_{\partial K_2\cap e}\odot \bm n_2,
\end{equation*}
where $\bm v \odot \bm n=\frac{1}{2}(\bm{v} \bm{n}^T+\bm n \bm v^T)$ is the 
symmetric part of the tensor product of $\bm v$ and $\bm n$.

When $e \in  \mathcal{E}_h^{B}$ then the above quantities are defined as
\[
\{\bm v\}=\bm v |_{e}\cdot \bm n,\quad
\{\bm \tau\}=\bm \tau|_{e}\bm n, \quad
[q]=q|_{e}, ~~ [\bm v]=\bm v|_{e}, \quad  \Lbrack\bm v\Rbrack=\bm v|_{e}\odot \bm n.
\]
If $\bm n_K$ is the outward unit normal to $\partial K$, it is easy to show that
\begin{equation}\label{eq:68}
\sum_{K\in \mathcal{T}_h}\int_{\partial K}\bm v\cdot \bm n_K q ds=\sum_{e\in \mathcal{E}_h}\int_{e}\{\bm v\} [q] ds,
\quad\mbox{for all}\quad \bm v\in H(\div ;\Omega),\quad\mbox{for all}\quad q\in H^1(\mathcal{T}_h).
\end{equation}
Also, for $\bm \tau \in H^1(\Omega)^{d\times d}$ and for all $\bm v\in H^1(\mathcal{T}_h)^d$, we have
\begin{equation}\label{eq:69}
\sum_{K\in \mathcal{T}_h}\int_{\partial K}(\bm \tau\bm n_K)\cdot  \bm vds=\sum_{e\in \mathcal{E}_h}\int_{e}\{\bm \tau\}\cdot [\bm v] ds.
\end{equation}

\noindent
The finite element spaces we consider are denoted by
\begin{eqnarray*}
\bm U_h&=&\{\bm u \in H(\div ;\Omega):\bm u|_K \in \bm U(K),~K \in \mathcal{T}_h;~ \bm u \cdot
\bm n=0~\hbox{on}~\partial \Omega\},
\\[1ex]
\bm V_{i,h}&=&\{\bm v \in H(\div ;\Omega):\bm v|_K \in \bm V_i(K),~K \in \mathcal{T}_h;~ \bm v \cdot
\bm n=0~\hbox{on}~\partial \Omega\},~~i=1,\dots,n,
\\
P_{i,h}&=&\{q \in L^2(\Omega):q|_K \in Q_i(K),~K \in \mathcal{T}_h; ~\int_{\Omega}q dx=0\},~~i=1,\dots,n.
\end{eqnarray*}
The discretizations that we analyze in the present context define the local spaces $\bm U(K)/\bm V_i(K)/Q_i(K)$
via the triplets $BDM_l(K)/$ $RT_{l-1}(K)/P_{l-1}(K)$, or
$BDFM_l(K)/RT_{l-1}(K)/P_{l-1}(K)$ for $l\ge 1$. Note that for each of these choices, the important 
condition $\div  \bm U(K)=\div  \bm V_i(K)=Q_i(K)$ is satisfied.

We recall the following basic approximation properties of these spaces: For all
$K\in \mathcal{T}_h$ and for all $\bm u \in H^s(K)^d$, there exists $\bm u_I\in
\bm U(K)$ such that
\begin{equation}\label{eq:70}
\|\bm u-\bm u_I\|_{0,K}+h_K|\bm u-\bm u_I|_{1,K}+h_K^2|\bm u-\bm
u_I|_{2,K}
\leq C h_K^s|\bm u|_{s,K}, ~2\le s \leq l+1.
\end{equation}
\subsection{DG discretization}
We note that according to the definition of $\bm U_h$, the normal component of any $\bm u\in \bm U_h$ is continuous
on the internal edges and vanishes on the boundary edges. Therefore, by splitting a vector  $\bm u\in \bm U_h$ into
its  normal and tangential components $\bm u_n$ and $\bm u_t$,
\begin{equation}\label{eq:71}
\bm u_n:=(\bm u\cdot \bm n)\bm n,\quad \bm u_t:=\bm u- \bm u_n,
\end{equation}
we have 
\begin{equation}\label{eq:72}
\mbox{for all}\quad e\in \mathcal{E}_h~~\int_e[\bm u_n]\cdot \bm \tau ds=0,\quad\mbox{for all}\quad \bm \tau\in H^1(\mathcal{T}_h)^d, \bm u\in \bm U_h,
\end{equation}
implying that
\begin{equation}\label{eq:73}
\mbox{for all}\quad e\in \mathcal{E}_h~~\int_e[\bm u]\cdot \bm \tau ds=\int_e[\bm u_t]\cdot \bm \tau ds,\quad\mbox{for all}\quad \bm  \tau\in H^1(\mathcal{T}_h)^d,\bm u \in \bm U_h.
\end{equation}
A direct computation 
shows that
\begin{equation}\label{eq:74}
\Lbrack \bm u_t\Rbrack:\Lbrack \bm w_t\Rbrack=\frac12[\bm u_t] \cdot [\bm w_t].
\end{equation}

Similar to the continuous problem, we denote 
$$
\bm v_h^T=(\boldsymbol v^T_{1,h}, \cdots \boldsymbol v^T_{n,h}), \quad \bm p_h^T=(p_{1,h},\cdots, p_{n,h}), 
\quad \bm z_h^T=(\boldsymbol z^T_{1,h}, \cdots \boldsymbol z^T_{n,h}),
$$
$$
\bm q_h^T=(q_{1,h},\cdots, q_{n,h}), \quad \bm V_h=\boldsymbol V_{1,h}\times\cdots\times\boldsymbol V_{n,h},
\quad \bm P_h= P_{1,h}\times\cdots\times P_{n,h}.
$$
With this notation at hand, the discretization of the variational problem~\eqref{eq:8} is given as follows:
Find $(\boldsymbol u_h;\bm v_h;\bm p_h, )\in \boldsymbol U_h\times \bm V_h\times \bm P_h$, such that for
any $(\boldsymbol w_h; \bm z_{h};\bm q_{h})\in \boldsymbol U_h\times\boldsymbol V_{h}\times \bm P_{h}$
\begin{subequations}\label{eq:75-77}
\begin{eqnarray}
 a_h(\bm u_h,\bm w_h) +\lambda ( \div \boldsymbol  u_h, \div \boldsymbol  w_h)- \sum_{i=1}^n(p_{i,h}, \div \boldsymbol  w_h)&=&(\boldsymbol f, \boldsymbol w_h),\label{eq:75}\\
  (R^{-1}_i\bv_{i,h},\bz_{i,h}) {-} (p_{i,h},\divv \bz_{i,h}) &=& 0,~~ i=1,\dots, n,  \label{eq:76} \\
 	 -(\divv\bu_h,q_{i,h}) - (\divv\bv_{i,h},q_{i.h})  +\tilde{\alpha}_{ii} (p_{i,h},q_{i,h}) +\sum_{\substack{{j=1}\\j\neq i}}^{n}\alpha_{ij} (p_{j,h},q_{i,h})&=& (g_i,q_{i,h}), i=1,\dots, n, \label{eq:77}
\end{eqnarray}
\end{subequations}
where
\begin{eqnarray}
a_h(\bm u,\bm w)&=&\label{78}
\sum _{K \in \mathcal{T}_h} \int_K \ep(\bm{u}) :
\ep(\bm{w}) dx-\sum_{e \in \mathcal{E}_h} \int_e \{\ep(\bm{u})\} \cdot [\bm w_t] ds\\
&&\nonumber-\sum _{e \in \mathcal{E}_h} \int_e \{\ep(\bm{w})\} \cdot [\bm u_t]ds+\sum _{e
\in \mathcal{E}_h} \int_e \eta h_e^{-1}[ \bm u_t] \cdot [\bm w_t] ds,
\end{eqnarray}
$\tilde{\alpha}_{ii}=-\alpha_{p_i}-\alpha_{ii}$, and $\eta $ is a stabilization parameter independent of parameters $\lambda,\,R_i^{-1},\, \alpha_{p_i},\, {\alpha}_{ij}$, $i,j=1,\dots,n$, the network scale $n$ and the mesh size $h$.
{
\begin{remark}
Consider the general rescaled boundary conditions 
\begin{subequations}\label{eq:Biot_BC}
\begin{eqnarray}
p_i&=& p_{i,D} \qquad \mbox{on }~~ \Gamma_{p_i,D},~~ i=1,\dots,n,\\ 
\bm v_i\cdot {\bm n} &=& q_{i,N}  \qquad \mbox{on}~~~ \Gamma_{p_i,N}, ~~i=1,\dots,n,\\ 
\u&=& {\bu}_D  \quad ~~~~\mbox{on}~~~\Gamma_{\vek{u},D}, \\ 
{(\bm \sigma -\sum_{i=1}^np_i\boldsymbol I){\vek{n}} }&=& {\bm g}_N \qquad~\,\mbox{on}~~~\Gamma_{\vek{u},N}.
\end{eqnarray}
\end{subequations}
Usually, it is assumed that the measure of $\Gamma_{\bm u, D}$ is nonzero to
guarantee the discrete Korn's inequality \cite{brenner2004korn}.

The standard way to incorporate the boundary conditions \eqref{eq:Biot_BC} is to modify the trial spaces according to the boundary conditions,
i.e., to seek the solution in the spaces
\begin{eqnarray*}
\bm U_h^D&=&\{\bm u \in H(\div ;\Omega):\bm u|_K \in \bm U(K),~K \in \mathcal{T}_h;~ \bm u \cdot
\bm n=\bm u_{D}\cdot \bm n~\hbox{on}~\Gamma_{\vek{u},D}\},
\\[1ex]
\bm V_{i,h}^D&=&\{\bm v \in H(\div ;\Omega):\bm v|_K \in \bm V_i(K),~K \in \mathcal{T}_h;~ \bm v \cdot
\bm n=q_{i,N}~\hbox{on}~\Gamma_{p_i,N}\},  i=1,\dots,n,\\
P_{i,h}&=&
\left\{
\begin{array}{l}
\{q \in L^2(\Omega):q|_K \in Q_i(K),~K \in \mathcal{T}_h,~~\hbox{if}~~|\Gamma_{p_i,D}| \neq 0\},\\
\{q \in L^2_0(\Omega):q|_K \in Q_i(K),~K \in \mathcal{T}_h,~~\hbox{if}~~ \Gamma_{p_i,N}=\Gamma \},
\end{array}
\right.  
i=1,\dots,n,
\end{eqnarray*}
and use the test spaces given by
\begin{eqnarray*}
\bm U_h^0&=&\{\bm u \in H(\div ;\Omega):\bm u|_K \in \bm U(K),~K \in \mathcal{T}_h;~ \bm u \cdot
\bm n=0~\hbox{on}~\Gamma_{\vek{u},D}\},
\\[1ex]
\bm V_{i,h}^0&=&\{\bm v \in H(\div ;\Omega):\bm v|_K \in \bm V_i(K),~K \in \mathcal{T}_h;~ \bm v \cdot
\bm n=0~\hbox{on}~\Gamma_{p_i,N}\},~i=1,\dots,n.
\end{eqnarray*}
Again denote $\boldsymbol V_h^D=\bm V_{1,h}^D\times\cdots\times\bm V_{n,h}^D$,
$\bm P_h=P_{1,h}\times\cdots\times P_{n,h}$,
$\boldsymbol V_h^0=\bm V_{1,h}^0\times\cdots\times\bm V_{n,h}^0$. 

Hence, problem~\eqref{eq:75-77} has the more general formulation:
Find $(\boldsymbol u_h; \boldsymbol v_h; \bm p_h)\in \boldsymbol U_h^D\times\boldsymbol V_h^D\times \bm P_h$, such that for
any $(\boldsymbol w_h; \boldsymbol z_h; \bm q_h)\in \boldsymbol U_h^0\times\boldsymbol V_h^0\times \bm P_h$
\begin{subequations}\label{general:boundary}
\begin{eqnarray}
 a_h(\bm u_h,\bm w_h) +\lambda ( \div \boldsymbol  u_h, \div \boldsymbol  w_h)- \sum_{i=1}^n(p_{i,h}, \div \boldsymbol  w_h)&=&F(
 \bm w_h),\\
  (R^{-1}_i\bv_{i,h},\bz_{i,h}) {-} (p_{i,h},{\rm div} \bz_{i,h}) =(p_{i,D},\bm z_{i,h}\cdot\bm n)_{\Gamma_{p_i,D}},&& i=1,\dots, n, \\
-({\rm div} \bu_h,q_{i,h}) - ({\rm div} \bv_{i,h},q_{i.h})  +\tilde{\alpha}_{ii} (p_{i,h},q_{i,h}) +\sum_{\substack{{j=1}\\j\neq i}}^{n}\alpha_{ij} (p_{j,h},q_{i,h})&=& (g_i,q_{i,h}), i=1,\dots, n, 
\end{eqnarray}
\end{subequations}
where
\begin{eqnarray}
a_h(\bm u,\bm w)&=&
\sum _{K \in \mathcal{T}_h} \int_K \ep(\bm{u}) :
\ep(\bm{w}) dx-\sum_{e \in \mathcal{E}_h^I\cup \mathcal{E}_h^{\bm u, D}} \int_e \{\ep(\bm{u})\} \cdot [\bm w_t] ds\\
&&\nonumber-\sum _{e \in \mathcal{E}_h^I\cup\mathcal{E}_h^{\bm u, D}} \int_e \{\ep(\bm{w})\} \cdot [\bm u_t]ds+\sum _{e
\in \mathcal{E}_h^I\cup \mathcal{E}_h^{\bm u, D}} \int_e \eta h_e^{-1}[ \bm u_t] \cdot [\bm w_t] ds,\\
F(\bm w)&=&(\bm f, \bm w)+({\bm g_N}, \bm w)_{\Gamma_{\vek{u},N}}-(\bm u_{D,t}, \ep(\bm{w})\bm n)_{\Gamma_{\vek{u},D}}+\sum _{e
\in \mathcal{E}_h^{\bm u, D}} \int_e \eta h_e^{-1} \bm u_{D,t} \cdot \bm w_t ds,
\end{eqnarray}
and $\bm u_{D,t}= \bm u_{D}- (\bm u_{D}\cdot \bm n) \bm n$, $\mathcal{E}_h^{\bm u, D}=\mathcal{E}_h^{B}\cap \Gamma_{\bm u, D}$, and $\eta$
is again a stabilization parameter which is independent of $\lambda,\,R_i^{-1},\, \alpha_{p_i},\, {\alpha}_{ij}$, $i,j=1,\dots,n$, the network scale $n$
and the mesh size $h$.

If $\Gamma_{\bm u, D}=\Gamma_{p,N}=\Gamma$
and $\bm u_D=0, q_N=0$, then \eqref{general:boundary} reduces to \eqref{eq:75-77} which will be analyzed in the remainder of this paper. If the measure
of $\Gamma_{\bm u, N}$ is nonzero, then the analysis is similar. If $\Gamma_{\bm u, D}=\Gamma$ and the measure of any $\Gamma_{p_i,D}, i=1,\cdots, n$, is nonzero,
then one has to modify the norms according to Remark~3.1 in \cite{HongKraus2017parameter}.
This part of the analysis is left as future work. 
\end{remark}
}

\begin{proposition}\label{prop:mass_cons}
Let $(\boldsymbol u_h; \boldsymbol v_h; \bm p_h)\in \boldsymbol U_h\times\boldsymbol V_h\times \bm P_h$ be the solution
of \eqref{eq:75}-\eqref{eq:77}, then $(\boldsymbol u_h; \boldsymbol v_h; \bm p_h)$ satisfy the pointwise mass conservation
equation
\begin{equation}\label{eq:conservative}
-{\rm div} \bu_h - {\rm div}\bv_{i,h}  -(\alpha_{p_i} +\alpha_{ii})p_{i,h} + \sum_{\substack{{j=1}\\j\neq i}}^{n}{\alpha}_{ij} p_{j,h} = Q_{i,h}g_i,~i=1,\dots,n ,~\forall x\in K, \forall K\in  \mathcal{T}_h,
\end{equation}
where $Q_{i,h}$ denotes the $L^2$-projection on $P_{i,h}$.

Furthermore, if $g_i=0$, then $-{\rm div} \bu_h - {\rm div}\bv_{i,h}  -(\alpha_{p_i} +\alpha_{ii})p_{i,h} + \sum\limits_{\substack{{j=1}\\j\neq i}}^{n}{\alpha}_{ij} p_{j,h}=0$.
\end{proposition}

For any
%$\bm u\in  H^1(\mathcal{T}_h)^d$,
$\bm u \in \boldsymbol U_h$, 
we introduce the mesh dependent norms:
\begin{eqnarray*}
\|\bm{u}\|_h^2&=&\sum _{K \in \mathcal{T}_h} 
\|\ep(\bm{u})\|_{0,K}^2+\sum _{e \in \mathcal{E}_h} h_e^{-1}\|[ \bm u_t]\|_{0,e}^2, \\
\|\bm u\|_{1,h}^2&=&\sum _{K \in \mathcal{T}_h} \|\nabla\bm{u}\|_{0,K}^2+\sum _{e \in \mathcal{E}_h} h_e^{-1}\|[ \bm{u}_t]\|_{0,e}^2 .
\end{eqnarray*}
Next, for $\bm u\in  \boldsymbol U_h$, we define the ``DG''-norm
\begin{equation}\label{DGnorm}
\|\bm u\|^2_{DG}=\sum _{K \in \mathcal{T}_h} \|\nabla\bm{u}\|_{0,K}^2+\sum _{e \in \mathcal{E}_h} h_e^{-1}\|[ \bm{u}_t]\|_{0,e}^2+\sum _{K \in \mathcal{T}_h}h_K^2|\bm{u}|^2_{2,K},
\end{equation}
and, finally, the mesh-dependent norm $\|\bm \cdot \|_{\bm U_h}$ by
\begin{equation}\label{U_hnorm}
\|\bm u\|^2_{\bm U_h}=\|\bm u\|^2_{DG}+\lambda \|\div \bm u\|^2.
\end{equation}
We now summarize several results on well-posedness and approximation
properties of the DG formulation, see, e.g.~\cite{hong2016robust, honguniformly}:

\begin{itemize}
\item From the discrete version of Korn's inequality we have that the norms $\|\cdot\|_{DG}$,
$\|\cdot\|_h$, and $\|\cdot\|_{1,h}$ are equivalent on $\bm U_h$, namely,
\begin{equation}\label{Korninequality}
\|\bm{u}\|_{DG}\eqsim  \|\bm{u}\|_h\eqsim\|\bm u\|_{1,h},\quad\mbox{for all}\quad~\bm u \in \bm U_h .
\end{equation}
\item The bilinear form $a_h(\cdot,\cdot)$,
introduced in~\eqref{78} is continuous and we have
\begin{eqnarray}\label{continuity:a_h}
|a_h(\bm u,\bm w)|&\lesssim& \| \bm u  \|_{DG}  \| \bm w  \|_{DG},\quad\mbox{for all}\quad \bm u,~\bm w\in H^2(\mathcal{T}_h)^d.
\end{eqnarray}
\item For our choice of the finite element spaces { $\bm U_h, \bm{V}_h$ and $\bm P_h$} we
have the following inf-sup conditions, see, e.g.~\cite{schtzau2002mixed}:
\begin{equation}\label{inf-sup}
\begin{split}
  \inf_{(q_{1,h},\cdots,q_{n,h})\in (P_{1,h}\cdots P_{n,h})}\sup_{\bm{u}_h\in \bm{U}_h}\frac{(\operatorname{div}\bm{u}_h,\sum\limits_{i=1}^n q_{i,h})}{\|\bm{u}_h\|_{1,h}\|\sum\limits_{i=1}^n q_{i,h}\|}&\geq \beta_{sd},\\
  \inf_{q_{i,h}\in P_{i,h}}\sup_{\bm{v}_{i,h}\in \bm{V}_{i,h}}\frac{(\operatorname{div}\bm{v}_{i,h},q_{i,h})}{\|\bm{v}_{i,h}\|_{\div}\|q_{i,h}\|}&\geq \beta_{dd},~i=1,\dots,n,
\end{split}
\end{equation}
where $\beta_{sd}$ and $\beta_{dd}$ are positive constant independent of the parameters $\lambda,R_i^{-1}, \alpha_{p_i}, {\alpha}_{ij}, i,j=1,\dots,n$, the network scale $n$ and the mesh size $h$.
\item We also have that $a_h(\cdot,\cdot)$ is coercive, and the proof
  of this fact parallels the proofs of similar results:
\begin{equation}\label{coercivity:a_h}
a_h(\bm{u}_h,\bm{u}_h)\geq \alpha_a \|\bm{u}_h\|^2_h,\quad\mbox{for all}\quad~\bm{u}_h\in\bm{U}_h,
\end{equation}
where $\alpha_a$ is a positive constant independent of parameters $\lambda,R_i^{-1}, \alpha_{p_i}, {\alpha}_{ij}, i,j=1,\dots,n$, the network scale $n$ and the mesh size $h$.
\end{itemize}

Related to the discrete problem~\eqref{eq:75}-\eqref{eq:77} we introduce the bilinear form 
\begin{equation}\label{eq:79}
\begin{split}
&\mathcal A_h((\bu_h; \bv_{h}; \bm p_{h}),(\bw_h; \bz_{h}; \bm q_{h}))\\
&=a_h(\bm u_h,\bm w_h)+\lambda ( \div \boldsymbol  u_h, \div \boldsymbol  w_h) -\sum_{i=1}^{n} (p_{i,h},{\rm div} \bw_h) 
+\sum_{i=1}^{n}(R^{-1}_i\bv_{i,h},\bz_{i,h}) - \sum_{i=1}^{n}(p_{i,h},{\rm div} \bv_{i,h})  \\
&-\sum_{i=1}^{n}({\rm div}\bu_h,q_{i,h}) -\sum_{i=1}^{n}({\rm div}\bv_{i,h},q_{i,h}) +\sum_{i=1}^{n}\tilde{\alpha}_{ii} (p_{i,h},q_{i,h}) +\sum_{i=1}^{n}\sum_{\substack{{j=1}\\j\neq i}}^{n}\alpha_{ij} (p_{j,h},q_{i,h}).
\end{split}
\end{equation}
%Let $\bm v_h=(\bv_{1,h},\dots,\bv_{n,h}), \bm p_h=(p_{1,h},\dots,p_{n,h}), \bm V_h=\bV_{1,h}\times\cdots\times \bV_{n,h}, \bm P_h=P_{1,h}\times\dots\times P_{n,h}$, then $A_h((\bu_h,\bv_{1,h},\dots,\bv_{n,h},p_{1,h},\dots,p_{n,h}),(\bw_h,\bz_{1,h},\dots,\bz_{n,h},q_{1,h},\dots,q_{n,h}))$ can be written as $\mathcal A_h((\bu_h;\bv_{h};\bm p_{h}),(\bw_h;\bz_{h};\bm q_{h}))$.

In view of the definitions of the norms $\|\cdot\|_{\bm U_h},\|\cdot\|_{\bm V}$ and $\|\cdot\|_{\bm P}$,
the boundedness of the bilinear form
$\mathcal A_h((\bu_h;\bv_{h};\bm p_{h}),(\bw_h;\bz_{h};\bm q_{h}))$
is obvious, i.e., the following theorem holds.
\begin{theorem}\label{boundd}
There exists a constant $C_{bd}$ independent of  the parameters $\lambda,R_i^{-1}, \alpha_{p_i}, {\alpha}_{ij}, i,j=1,\dots,n$,
the network scale $n$ and the mesh size $h$,
such that  for any $(\bm u_h; \bm v_h;\bm p_h)
\in \bm U_h\times\bm V_h\times \bm P_h, (\bm w_h; \bm z_h; \bm q_h)\in \bm U_h\times\bm V_h\times \bm P_h$
there holds
\begin{equation*}
|\mathcal A_h((\bu_h;\bv_{h};\bm p_{h}),(\bw_h;\bz_{h};\bm q_{h}))|\le C_{bd} (\|\bm u_h\|_{\bm U_h}+\|\bm v_h\|_{\bm V}
+\|\bm p_h\|_{\bm P})  (\|\bm w_h\|_{\bm U_h}+\|\bm z_h\|_{\bm V}+\|\bm q_h\|_{\bm P}).
\end{equation*}
\end{theorem}
We come to our second main result.

\begin{theorem}\label{Dis:Stability}
There exits a constant $\beta_0>0$ independent of  the parameters $\lambda,R_i^{-1}, \alpha_{p_i}, {\alpha}_{ij}, i,j=1,\dots,n$,
the network scale $n$ and the mesh size $h$, such that 
\begin{equation}\label{eq:80}
\resizebox{.935\hsize}{!}{$
\displaystyle\inf_{(\boldsymbol u_h;  \boldsymbol v_h; \bm p_h)\in \boldsymbol U_h\times\boldsymbol V_h\times \bm P_h} 
\sup_{(\boldsymbol w_h;\boldsymbol z_h; \bm q_h)\in \boldsymbol U_h\times\boldsymbol V_h\times \bm P_h}\frac{\mathcal A_h((\bu_h;\bv_{h};\bm p_{h}),(\bw_h;\bz_{h};\bm q_{h})))}{(\|\boldsymbol u_h\|_{\bm U_h}+
\|\boldsymbol v_h\|_{\boldsymbol V}+\|\bm p_h\|_{\bm P})(\|\boldsymbol w_h\|_{\bm U_h}+\|\boldsymbol z_h\|_{\boldsymbol V}+\|\bm q_h\|_{\bm P})}\geq \beta_0.
$}
\end{equation}
\end{theorem}
The proof of this theorem can be obtained by following the proof of Theorem~\ref{stability:continuous}
and using the technique shown in~\cite{HongKraus2017parameter}.

From the above theorem, we get the following stability estimate.
\begin{corollary}\label{eq:92}
Let $(\boldsymbol u_h;  \boldsymbol v_h; \bm p_h)\in \boldsymbol U_h\times \boldsymbol V_h\times \bm P_h$ be the solution of \eqref{eq:75}-\eqref{eq:77},
then we have the estimate
\begin{equation}
\|\boldsymbol u_h\|_{\boldsymbol U_h}+\|\boldsymbol  v_h\|_{\boldsymbol V}+\|\bm p_h\|_{\bm P}\leq C_2 (\|\boldsymbol f\|_{\boldsymbol U_h^*}+\|\bm g\|_{\bm P^*}),
\end{equation} 
where $\|\boldsymbol f\|_{\boldsymbol U_h^*}=
\sup\limits_{\boldsymbol w_h\in \boldsymbol U_h}\frac{(\boldsymbol f, \boldsymbol w_h)}
{\|\boldsymbol w_h\|_{\boldsymbol U_h}},\|\bm g\|_{\bm P^*}=\sup\limits_{\bm q_h\in \bm P_h}\frac{(\bm g, \bm q_h)}{\|\bm q_h\|_{\bm P}}$
and $C_2$ is a constant independent of $\lambda,\,R_i^{-1},\,\alpha_{p_i},\,{\alpha}_{ij}$, $i,j=1,\dots,n$, the network scale $n$ and the mesh size $h$.
\end{corollary}

\begin{remark}
Define 
\begin{equation}\label{Preconditioner:Bh}
\mathcal{B}_h:=\left[\begin{array}{ccc}
\mathcal{B}_{h,\bm u}^{-1} & \bm 0 & \bm 0 \\
\bm 0 & \mathcal{B}_{h,\bm v}^{-1}& \bm 0 \\
\bm 0& \bm 0& \mathcal{B}_{h,\bm p}^{-1}  
\end{array}
\right],
\end{equation}
where 
$$
\mathcal{B}_{h,\bm u}=-\div_h \ep_h-\lambda \nabla_h \div_h,
$$
\begin{align*}
\mathcal{B}_{h,\bm v}&=  
\begin{bmatrix}
R_1^{-1}I_h&0  & \dots &0  \\
0 &R_2^{-1} I_h& \dots &0 \\
\vdots &\vdots  & \ddots & \vdots  \\
0& 0 & \dots &R_n^{-1}I_h
\end{bmatrix}
-
\begin{bmatrix}
\tilde{\gamma}_{11}\nabla_h{\rm div}_h &\tilde{\gamma}_{12}\nabla_h{\rm div}_h& \dots &\tilde{\gamma}_{1n}\nabla_h{\rm div}_h\\
\tilde{\gamma}_{21}\nabla_h{\rm div}_h &\tilde{\gamma}_{22}\nabla_h{\rm div}_h & \dots &\tilde{\gamma}_{2n}\nabla_h{\rm div}_h\\
\vdots &\vdots  & \ddots & \vdots  \\
\tilde{\gamma}_{n1}\nabla_h{\rm div}_h&\tilde{\gamma}_{n2}\nabla_h{\rm div}_h & \dots &\tilde{\gamma}_{nn}\nabla_h{\rm div}_h
\end{bmatrix},
\end{align*}
\begin{align*}
\mathcal{B}_{h,\bm p}&=
\begin{bmatrix}
\gamma_{11}I_h&\gamma_{12}I_h& \dots &\gamma_{1n}I_h\\
\gamma_{21}I_h&\gamma_{22}I_h& \dots &\gamma_{2n}I_h\\
\vdots&\vdots  &\ddots&\vdots\\
\gamma_{n1}I_h&\gamma_{n2}I_h& \dots &\gamma_{nn}I_h\\
\end{bmatrix}.
\end{align*}

Then due to the theory presented in \cite{Mardal2011preconditioning}, Theorems~\ref{boundd} and~\ref{Dis:Stability} imply
that the norm-equivalent (canonical) block-diagonal preconditioner $\mathcal{B}_h$ for the operator 
\begin{equation}\label{operator:A_h}
\mathcal{A}_h:=
\begin{bmatrix}
- {\rm div}_h \ep_h  - \lambda \nabla_h {\rm div}_h     & 0  & \dots     & \dots&0      &   \nabla_h & \dots & \dots & \nabla_h    \\
\\
0      & R_1^{-1}I_h &     0      & \dots  &0         &  \nabla_h  &  0      & \dots  &0 \\
\vdots &  0           &     \ddots &     & \vdots       &  0           &     \ddots &     & \vdots \\
\vdots &      \vdots  &            &\ddots  & 0            &       \vdots  &            &\ddots  & 0 \\
0      &     0        &  \dots     & 0      &R_n^{-1}I_h&  0        &  \dots     & 0   & \nabla_h\\
\\
- {\rm div}_h   &- {\rm div}_h  &0 &\dots   &  0            &\tilde\alpha_{11}I_h& \alpha_{12} I_h & \dots& \alpha_{1n}I_h\\
\vdots                 &   0       &   \ddots        &                 & \vdots      &  \alpha_{21}I_h & \ddots & & \alpha_{2n}I_h\\
\vdots                 &  \vdots  &              & \ddots              & 0    &  \vdots &  & \ddots & \vdots\\
- {\rm div}_h   &   0    & \dots &   0    &-  {\rm div}_h   & \alpha_{n1}I_h  &  \alpha_{n2}I_h &   \dots   & \tilde\alpha_{nn} I_h  \\     
\end{bmatrix},
\end{equation}
induced by the bilinear form~\eqref{eq:79} is uniform with respect to variation of the model and dicretization parameters.

This means that the condition number $\kappa(\mathcal{B}_h \mathcal{A}_h)$ is uniformly bounded with respect to the
parameters $\lambda,R_i^{-1}, \alpha_{p_i}, {\alpha}_{ij}, i,j=1,\dots,n$ in the ranges specified in~\eqref{parameter:range},
the network scale $n$ and the mesh size $h$. 

To apply the preconditioner $\mathcal{B}_h$, one has to solve an elasticity system discretized by an $H(\div)$-conforming
discontinuous Galerkin method \cite{hong2016robust} and $n$ coupled elliptic $H(\div)$ problems discretized by $RT$ elements. 
In the lowest order case and for $n=1$, optimal solvers for this task have been proposed in~\cite{kraus2016heterogeneous}.
\end{remark}

\section{Error estimates}\label{sec:error_estimates}

In this section, we derive the error estimates that follow from the results presented in Section~\ref{sec:uni_stab_disc_model}.
{Let    
$\Pi_B^{\div}: H^1(\Omega)^d\mapsto \bm U_h$ be the canonical interpolation operator. We also denote the
$L^2$-projection on $P_{i,h}$ by~$Q_{i,h}$. The following Lemma, see \cite{hong2016robust}, summarizes some of the properties of
$\Pi_B^{ \div}$ and $Q_{i,h}$ needed for our proof.}
\begin{lemma}\label{stable:Interpolation}
For all $\bm w \in H^1(K)^d$ we have
\begin{eqnarray*}
 \div \Pi_B^{ \div}=Q_{i,h}  \div\;;\quad
|\Pi_B^{ \div} \bm w|_{1,K} \lesssim |\bm w|_{1,K};~  \|\bm w-\Pi_B^{ \div} \bm w\|^2_{0,\partial K}\lesssim h_K
  |\bm w|^2_{1,K}.
  %\quad 
  %\| \div(\bm w-\Pi_h^{ \div} \bm w)\|_{-1}\lesssim h_K
  %\| \div\bm w\|,
\end{eqnarray*}
%where $\|r\|_{-1} = \sup_{\chi\in H^1}\frac{(\chi,r)}{\|\chi\|_1}$. 
\end{lemma}

\begin{theorem}\label{error0}
Let $(\bm u;\bm v; \bm p)$ be the solution of~\eqref{eq:8} and $(\bm u_h;\bm v_h;\bm p_h)$
be the solution of \eqref{eq:75}--\eqref{eq:77}. Then the error estimates
\begin{equation} \label{eq:erroruv}
\|\bm u-\bm u_h\|_{\bm U_h}+\|\bm v-\bm v_h\|_{\bm V}\leq C_{e,u}  \inf\limits_{\boldsymbol w_h\in \boldsymbol U_h, \bm z_h\in \bm V_h}\Big(\|\bm u-\bm w_h\|_{\bm U_h}+\|\bm v-\bm z_h\|_{\bm V}\Big),
\end{equation}
and 
\begin{equation} \label{eq:errorp}
\|\bm p-\bm p_h\|_{\bm P}\leq C_{e,p}  \inf\limits_{\boldsymbol w_h\in \boldsymbol U_h, \bm z_h\in \bm V_h, \bm q_h\in {\bm P_h}}\Big(\|\bm u-\bm w_h\|_{\bm U_h}+\|\bm v-\bm z_h\|_{\bm V}+\|\bm p-\bm q_h\|_{\bm P}\Big),
\end{equation}
hold, where $C_{e,u}, C_{e,p}$ are constants independent of $\lambda,R_i^{-1}, \alpha_{p_i}, {\alpha}_{ij}, i,j=1,\dots,n$, the network scale $n$ and the mesh size $h$. 
\end{theorem}
\begin{proof}
Subtracting \eqref{eq:75}--\eqref{eq:77} from \eqref{eq:8a}--\eqref{eq:8c} and noting the consistency of $ a_h(\cdot,\cdot)$, we have that for any $(\boldsymbol w_h; \boldsymbol z_h; \bm q_h)\in \boldsymbol U_h\times\boldsymbol V_h\times \bm P_h$
\begin{eqnarray}
 a_h(\bm u-\bm u_h,\bm w_h) +\lambda ( \div (\boldsymbol u- \boldsymbol u_h), \div \boldsymbol  w_h)- ( \sum _{i=1}^n(p_i-p_{i,h}), \div \boldsymbol  w_h)&=&0,\label{eq:error1}\\
(R_i^{-1}(\boldsymbol v_i- \boldsymbol v_{i,h}), \boldsymbol  z_{i,h})- (p_i-p_{i,h}, \div \boldsymbol  z_{i,h})&=& 0,~~~i=1,\dots, n,\label{eq:error2} \\
\nonumber-( \div (\boldsymbol u- \boldsymbol u_h),q_{i,h})  -( \div (\boldsymbol v_i- \boldsymbol v_{i,h}),q_{i,h})   +\tilde{\alpha}_{ii} (p_{i}-p_{i,h},q_{i,h})\\
+\sum_{\substack{j=1\\j\neq i}}^{n}\alpha_{ij}(p_{j}-p_{j,h},q_{i,h})&=&0,~~~i=1,\dots, n. \label{eq:error3}
\end{eqnarray}
%We only consider the space choice of $BDM_l/RT_{l-1}/P_{l-1}$, the other cases are similar. 

Let $\boldsymbol u_I=\Pi_{B}^{\div}\boldsymbol u\in \boldsymbol U_h, p_{i,I}=Q_{i,h} p_i\in P_{i,h}$.
%where $\Pi_{B}^{div}$
%is the canonical interpolation operator from $\boldsymbol U$ to $\boldsymbol U_h$,  and $Q_h$ is the $L^2$ orthogonal
%projection from $L^2(\Omega)$ to $P_h$. 
Now for arbitrary $\boldsymbol v_{i,I}\in \bm V_{i,h}$,
%$\boldsymbol v_I=\pi_{R}^{div} \boldsymbol v \in \boldsymbol V_h$. $\pi_{R}^{div}$ is the projection operator introduced by Winther from $\boldsymbol V$ to $\boldsymbol V_h$
from~\eqref{eq:error1}--\eqref{eq:error3}, noting that $\div\Pi_{B}^{\div}=Q_{i,h}\div$
and $\div \boldsymbol U_h=\div \boldsymbol V_{i,h}=P_{i,h}$, we conclude
\begin{equation*}
\begin{split}
 \nonumber a_h(\bm u_I-\bm u_h,\bm w_h) +\lambda ( \div (\boldsymbol u_I- \boldsymbol u_h), \div \boldsymbol  w_h)
 - \sum_{i=1}^n(p_{i,I}-p_{i,h}, \div \boldsymbol  w_h)&=a_h(\bm u_I-\bm u,\bm w_h),\\
\nonumber (R_i^{-1}(\boldsymbol v_{i,I}- \boldsymbol v_{i,h}), \boldsymbol  z_{i,h})- (p_{i,I}-p_{i,h}, \div \boldsymbol  z_{i,h})&=(R_i^{-1}(\boldsymbol v_{i,I}- \boldsymbol v_i), \boldsymbol  z_{i,h}),i=1,\dots,n, \\
\nonumber-( \div (\boldsymbol u_I- \boldsymbol u_h),q_{i,h}) -( \div (\boldsymbol v_{i,I}- \boldsymbol v_{i,h}),q_{i,h}) 
+ \tilde{\alpha}_{ii} (p_{i,I}-p_{i,h},q_{i,h})
\nonumber\\+\sum_{\substack{j=1\\j\neq i}}^{n}\alpha_{ij}(p_{j,I}-p_{j,h},q_{i,h})&=-(\div (\boldsymbol v_{i,I}- \boldsymbol v_i),q_{i,h}),i=1,\dots,n.
\end{split}
\end{equation*}
Next, since $(\bm u_I-\bm u_h) \in \bm U_h$, $(\bm v_I-\bm v_h) \in \bm V_h$, $(\bm p_I-\bm p_h) \in \bm P_h$,
by the stability result~\eqref{eq:80} for the discrete problem~\eqref{eq:75}--\eqref{eq:77}, we obtain
\begin{equation*}%\label{eq:error7}
\begin{split}
&\|\bm u_I-\bm u_h\|_{\bm U_h}+\|\bm v_I-\bm v_h\|_{\bm V}\\
&\leq C_e \Big( \sup\limits_{\boldsymbol w_h\in \boldsymbol U_h}\frac{a_h(\bm u_I-\bm u,\bm w_h)}{\|\bm w_h\|_{\bm U_h}}+
\sup\limits_{\boldsymbol z_h\in \boldsymbol V_h}\frac{
\displaystyle\sum_{i=1}^n\big(R_i^{-1}(\boldsymbol v_{i,I}- \boldsymbol v_i), \boldsymbol  z_{i,h}\big)}{\|\bm z_h\|_{\bm V}}+\sup\limits_{\bm q_h\in \bm P_h}\frac{(\Divv (\boldsymbol v- \boldsymbol v_I),\bm q_h)}{\|\bm q_h\|_{\bm P}}\Big),
\end{split}
\end{equation*}
\begin{equation*}%\label{eq:error8}
\|\bm p_I-\bm p_h\|_P\leq C_e  \Big(\sup\limits_{\boldsymbol w_h\in \boldsymbol U_h}\frac{a_h(\bm u_I-\bm u,\bm w_h)}{\|\bm w_h\|_{\bm U_h}}+\sup\limits_{\boldsymbol z_h\in \boldsymbol V_h}
\frac{\displaystyle\sum_{i=1}^n\big(R_i^{-1}(\boldsymbol v_{i,I}- \boldsymbol v_i), \boldsymbol  z_{i,h}\big)}{\|\bm z_h\|_{\bm V}}+\sup\limits_{\bm q_h\in \bm P_h}\frac{(\Divv (\boldsymbol v- \boldsymbol v_I),\bm q_h)}{\|\bm q_h\|_{\bm P}}\Big).
\end{equation*}
Using the boundedness of $a_h(\cdot, \cdot)$, {the second inequality in Lemma \ref{stable:Interpolation}}, the triangle inequality
and noting that $\bm v_I$ is arbitrary and $(\Divv (\boldsymbol v- \boldsymbol v_I),\bm q_h)\le \|\boldsymbol v- \boldsymbol v_I\|_{\bm V}\|\bm q_h\|_{\bm P}$, we have that
\begin{equation} \label{eq:error9}
\|\bm u-\bm u_h\|_{\bm U_h}+\|\bm v-\bm v_h\|_{\bm V}\leq C_{e,u}  \inf\limits_{\boldsymbol w_h\in \boldsymbol U_h, \bm z_h\in \bm V_h}\Big(\|\bm u-\bm w_h\|_{\bm U_h}+\|\bm v-\bm z_h\|_{\bm V}\Big),
\end{equation}
and 
\begin{equation} \label{eq:error10}
\|\bm p-\bm p_h\|_P\leq C_{e,p}  \inf\limits_{\boldsymbol w_h\in \boldsymbol U_h, \bm z_h\in \bm V_h, \bm q_h\in {\bm P_h}}\Big(\|\bm u-\bm w_h\|_{\bm U_h}+\|\bm v-\bm z_h\|_{\bm V}+\|\bm p-\bm q_h\|_{\bm P}\Big).
\end{equation}
\end{proof}
\begin{remark}
From the above theorem, we can see that the discretizations are locking-free.
%when the material goes to nearly impressible.
\end{remark}

\section{Numerical Experiments}\label{sec:NumericalExperiments}

The following numerical experiments are for three widely applied MPET models, namely the one-network, two-network and
four-network models. 
We suppose that the domain $\Omega$ is the unit square in $\mathbb{R}^2$ and during the discretization it has been {partitioned as bisections of $2N^2$ triangles 
with mesh size $h=1/N$}. 
To discretize the pressure variables we use discontinuous piecewise constant elements, the fluxes are discretized employing the lowest-order Raviart-Thomas space and the displacement we 
approximate with the Brezzi-Douglas-Marini elements of lowest order. All the numerical tests included in this section have been carried out in FEniCS, \cite{AlnaesBlechta2015a,LoggMardalEtAl2012a}. 
The aim of these experiments is:
\begin{itemize}
\item[(i)] to validate the convergence of the error estimates in the derived parameter-dependent norms;
\item[(ii)] to test the robustness of the proposed block-diagonal preconditioners by using it within the MinRes algorithm.
\end{itemize}

\subsection{The one network model}

Here we consider the simplest case of a system with only one pressure and one flux, i.e., the Biot's consolidation model. We solve system~\eqref{eq:6}
for 
$$
\bf=\left(
\begin{array}{c}
-(2y^3-3y^2+y)(12x^2-12x+2)-(x-1)^2x^2(12y-6)+900(y-1)^2y^2(4x^3-6x^2+2x)\\
~~(2x^3-3x^2+x)(12y^2-12y+2)+(y-1)^2y^2(12x-6)+900(x-1)^2x^2(4y^3-6y^2+2y)
\end{array}
\right)
$$
and 
$$g=R_1\left(\frac{\partial \phi_2}{\partial x}+\frac{\partial \phi_2}{\partial y}\right)-\alpha_{p_1}(\phi_2-1),
$$
where $(x,y)\in \Omega$ and $\phi_1=(x-1)^2(y-1)^2 x^2 y^2$,  $\phi_2=900(x-1)^2(y-1)^2 x^2 y^2$. 

Then the exact solution of system~\eqref{eq:6} with boundary conditions
$
\boldsymbol u|_{\partial\Omega}=0,\; \boldsymbol v\cdot \boldsymbol n|_{\partial\Omega}=0
$
is given by
$
\boldsymbol u=\left(\frac{\partial \phi_1}{\partial y},-\frac{\partial \phi_1}{\partial x} \right),\; p=\phi_2-1,\; \boldsymbol v=-R_1\nabla p
$
and $p\in L^2_0(\Omega)$.

We performed experiments with different sets of input parameters. In Tables~\ref{tab1}--\ref{tab3} we report the error of the
numerical solution in the introduced parameter-dependent norms
$\Vert \cdot \Vert_{\bm P}$, $\Vert \cdot \Vert_{\bm V}$, $\Vert \cdot \Vert_{\bm U_h}$. Additionally, we list the number of
MinRes iterations $n_{it}$ and average residual convergence factor with the proposed block-diagonal preconditioner where the
stopping criterion is residual reduction by $10^8$ in the norm induced by the preconditioner.
The robustness of the method is validated with respect to variation of the parameters $\lambda$, $R_1^{-1}$, $\alpha_{p_1}$, 
as introduced in~\eqref{eq:6}, and the discretization parameter $h$. 

\begin{table}[h!]
\caption{Errors measured in parameter-dependent norms ($\alpha_{p_1}=10^{-4}$, $\lambda=10^4$).}
\label{tab1}
\begin{center}
\begin{tabular}{c|c|c|c|c|c|c|c}
& &  \multicolumn{6}{c}{$R_{1}^{-1}$} \\ [0.7ex] \hline
$h$ & &  1E0 & 1E2 & 1E3 & 1E4 & 1E8  & 1E16 \\ [0.7ex] \hline
\multirow{3}{*}{$\displaystyle \frac{1}{8}$} & $\Vert \cdot\Vert_{\bm P}$ & 2.1E--1 & 2.1E--2 & 6.6E--3 & 2.1E--3 & 2.0E--3 & 2.0E--3  \\ [0.7ex]
& $\Vert \cdot\Vert_{\bm V}$ & 1.3E1 & 1.3E0 & 4.1E--1 & 1.3E--1 & 1.6E--4 & 1.6E--8  \\ [0.7ex]
& $\Vert \cdot\Vert_{\bm U_h}$ & 9.1E--2 & 9.1E--2 & 9.1E--2 & 9.1E--2 & 9.1E--2 & 9.1E--2  \\[0.7ex] \hline
\multirow{3}{*}{$\displaystyle \frac{1}{16}$} & $\Vert \cdot\Vert_{\bm P}$ & 1.0E--1 & 1.0E--2 & 3.3E--3 & 1.0E--3 & 1.0E--3 & 1.0E--3  \\ [0.7ex]
& $\Vert \cdot\Vert_{\bm V}$ & 6.6E0 & 6.6E--1 & 2.1E--1 & 6.6E--2 & 8.3E--5 & 8.3E--9  \\ [0.7ex]
& $\Vert \cdot\Vert_{\bm U_h}$ & 4.5E--2 & 4.5E--2 & 4.5E--2 & 4.5E--2 & 4.5E--2 & 4.5E--2  \\ [0.7ex] \hline
\multirow{3}{*}{$\displaystyle \frac{1}{32}$} & $\Vert \cdot\Vert_{\bm P}$ & 5.2E--2 & 5.1E--3 & 1.6E--3 & 5.1E--4 & 5.1E--4 & 5.2E--4  \\ [0.7ex]
& $\Vert \cdot\Vert_{\bm V}$ & 3.3E0 & 3.3E--1 & 1.0E--1 & 3.3E--2 & 4.4E--5 & 4.4E--9  \\ [0.7ex]
& $\Vert \cdot\Vert_{\bm U_h}$ & 2.3E--2 & 2.3E--2 & 2.3E--2 & 2.3E--2 & 2.3E--2 & 2.3E--2  \\ [0.7ex]  \hline
\multirow{3}{*}{$\displaystyle \frac{1}{64}$} & $\Vert \cdot\Vert_{\bm P}$ & 2.6E--2 & 2.6E--3 & 8.2E--4 & 2.6E--4 & 2.6E--4 & 2.6E--4  \\ [0.7ex]
& $\Vert \cdot\Vert_{\bm V}$ & 1.7E0 & 1.7E--1 & 5.2E--2 & 1.7E--2 & 2.3E--5 & 2.3E--9  \\ [0.7ex]
& $\Vert \cdot\Vert_{\bm U_h}$ & 1.1E--2 & 1.1E--2 & 1.1E--2 & 1.1E--2 &  1.1E--2 & 1.1E--2   \\  [0.7ex] \hline
\multirow{3}{*}{$\displaystyle \frac{1}{128}$} & $\Vert \cdot\Vert_{\bm P}$ & 1.3E--2 & 1.3E--3 & 4.1E--4 & 1.3E--4 & 1.3E--4 & 1.3E--4  \\ [0.7ex]
& $\Vert \cdot\Vert_{\bm V}$ & 8.2E--1 & 8.2E--2 & 2.6E--2 & 8.2E--3 & 1.2E--5 & 1.2E--9  \\[0.7ex]
& $\Vert \cdot\Vert_{\bm U_h}$ & 5.6E--3 & 5.6E--3 & 5.6E--3 & 5.6E--3 & 5.6E--3 & 5.6E--3  \\[0.7ex] \hline
\multirow{3}{*}{$\displaystyle \frac{1}{256}$} & $\Vert \cdot\Vert_{\bm P}$ & 6.6E--3 & 6.6E--4 & 2.1E--4 & 6.6E--5 & 6.6E--5 & 6.6E--5  \\ [0.7ex]
& $\Vert \cdot\Vert_{\bm V}$ & 4.1E--1 & 4.1E--2 & 1.3E--2 & 4.1E--3 & 6.1E--6 & 6.1E--10  \\[0.7ex]
& $\Vert \cdot\Vert_{\bm U_h}$ & 2.8E--3 & 2.8E--3 & 2.8E--3 & 2.8E--3 & 2.8E--3 & 2.8E--3
\end{tabular}
\end{center}
\end{table}
\begin{table}[h!]
\caption{Errors measured in parameter-dependent norms ($\alpha_{p_1}=0$, $R_1^{-1}=10^8$).}
\label{tab2}
\begin{center}
\begin{tabular}{c|c|c|c|c}
& &  \multicolumn{3}{c}{$\lambda$} \\[0.7ex] \hline
$h$ & &  1E0 & 1E4 & 1E8  \\ [0.7ex] \hline
\multirow{3}{*}{$\displaystyle \frac{1}{8}$} & $\Vert \cdot\Vert_{\bm P}$ & 2.0E--1 & 2.0E--3 & 2.1E--5   \\[0.7ex]
& $\Vert \cdot\Vert_{\bm V}$ & 1.6E--4 & 1.6E--4 & 1.3E--3   \\[0.7ex]
& $\Vert \cdot\Vert_{\bm U_h}$ & 9.1E--2 & 9.1E--2 & 9.1E--2   \\ [0.7ex]\hline
\multirow{3}{*}{$\displaystyle \frac{1}{16}$} & $\Vert \cdot\Vert_{\bm P}$ & 1.0E--1 & 1.0E--3 & 1.0E--5   \\[0.7ex]
& $\Vert \cdot\Vert_{\bm V}$ & 8.9E--5 & 8.6E--5 & 6.5E--4   \\ [0.7ex]
& $\Vert \cdot\Vert_{\bm U_h}$ & 4.5E--2 & 4.5E--2 & 4.5E--2   \\[0.7ex] \hline
\multirow{3}{*}{$\displaystyle \frac{1}{32}$} & $\Vert \cdot\Vert_{\bm P}$ & 5.2E--2 & 5.2E--4 & 5.2E--6   \\[0.7ex]
& $\Vert \cdot\Vert_{\bm V}$ & 5.7E--5 & 4.5E--5 & 3.3E--4   \\[0.7ex]
& $\Vert \cdot\Vert_{\bm U_h}$ & 2.3E--2 & 2.3E--2 & 2.3E--2   \\[0.7ex] \hline
\multirow{3}{*}{$\displaystyle \frac{1}{64}$} & $\Vert \cdot\Vert_{\bm P}$ & 2.6E--2 & 2.6E--4 & 2.6E--6   \\[0.7ex]
& $\Vert \cdot\Vert_{\bm V}$ & 4.6E--5 & 2.3E--5 & 1.6E--4   \\[0.7ex]
& $\Vert \cdot\Vert_{\bm U_h}$ & 1.1E--2 & 1.1E--2 &  1.1E--2   \\ [0.7ex] \hline
\multirow{3}{*}{$\displaystyle \frac{1}{128}$} & $\Vert \cdot\Vert_{\bm P}$ & 1.3E--2 & 1.3E--4 & 1.3E--6   \\ [0.7ex]
& $\Vert \cdot\Vert_{\bm V}$ & 4.3E--5 & 1.2E--5 & 8.2E--5   \\ [0.7ex]
& $\Vert \cdot\Vert_{\bm U_h}$ & 5.6E--3 & 5.6E--3 & 5.6E--3   \\ [0.7ex] \hline
\multirow{3}{*}{$\displaystyle \frac{1}{256}$} & $\Vert \cdot\Vert_{\bm P}$ & 6.6E--3 & 6.6E--5 & 6.6E--7   \\ [0.7ex]
& $\Vert \cdot\Vert_{\bm V}$ & 4.1E--5 & 6.1E--6 & 4.1E--5   \\ [0.7ex]
& $\Vert \cdot\Vert_{\bm U_h}$ & 2.8E--3 & 2.8E--3 & 2.8E--3  
\end{tabular}
\end{center}
\end{table}

\begin{table}[h!]
\caption{Errors measured in parameter-dependent norms ($R_1^{-1}=10^{4}$, $\lambda=10^0$).}
\label{tab3}
\begin{center}
\begin{tabular}{c|c|c|c|c|c}
%\multicolumn{6}{c}{$R_1^{-1}=10^{4}$, $\lambda=10^0$} \\
& &  \multicolumn{4}{c}{$\alpha_{p_1}$} \\[0.7ex] \hline
$h$ & &  1E0 & 1E--4 & 1E--8 &  0  \\ [0.7ex]\hline
\multirow{3}{*}{$\displaystyle \frac{1}{8}$} & $\Vert \cdot\Vert_{\bm P}$ & 2.0E--1 & 2.0E--1 & 2.0E--1 & 2.0E--1  \\[0.7ex]
& $\Vert \cdot\Vert_{\bm V}$ & 1.6E--2 & 1.6E--2 & 1.6E--2 & 1.6E--2   \\ [0.7ex]
& $\Vert \cdot\Vert_{\bm U_h}$ & 9.0E--2 & 9.1E--2 & 9.1E--2 & 9.1E--2   \\[0.7ex] \hline
\multirow{3}{*}{$\displaystyle \frac{1}{16}$} & $\Vert \cdot\Vert_{\bm P}$ & 1.0E--1 & 1.0E--1 & 1.0E--1 & 1.0E--1   \\[0.7ex]
& $\Vert \cdot\Vert_{\bm V}$ & 8.1E--3 & 8.3E--3 & 8.3E--3 & 8.3E--3   \\[0.7ex]
& $\Vert \cdot\Vert_{\bm U_h}$ & 4.5E--2 & 4.5E--2 & 4.5E--2 & 4.5E--2   \\[0.7ex] \hline
\multirow{3}{*}{$\displaystyle \frac{1}{32}$} & $\Vert \cdot\Vert_{\bm P}$ & 5.2E--2 & 5.2E--2 & 5.2E--2 & 5.2E--2   \\[0.7ex]
& $\Vert \cdot\Vert_{\bm V}$ & 4.1E--3 & 4.2E--3 & 4.2E--3 & 4.2E--3   \\ [0.7ex]
& $\Vert \cdot\Vert_{\bm U_h}$ & 2.2E--2 & 2.2E--2 & 2.2E--2 & 2.2E--2   \\ [0.7ex] \hline
\multirow{3}{*}{$\displaystyle \frac{1}{64}$} & $\Vert \cdot\Vert_{\bm P}$ & 2.6E--2 & 2.6E--2 & 2.6E--2 & 2.6E--2   \\[0.7ex]
& $\Vert \cdot\Vert_{\bm V}$ & 2.0E--5 & 2.1E--3 & 2.1E--3 & 2.1E--3   \\ [0.7ex]
& $\Vert \cdot\Vert_{\bm U_h}$ & 1.1E--2 & 1.1E--2 &  1.1E--2 & 1.1E--2   \\ [0.7ex] \hline
\multirow{3}{*}{$\displaystyle \frac{1}{128}$} & $\Vert \cdot\Vert_{\bm P}$ & 1.3E--2 & 1.3E--3 & 1.3E--3 & 1.3E--3   \\[0.7ex]
& $\Vert \cdot\Vert_{\bm V}$ & 1.0E--5 & 1.0E--5 & 1.0E--5 & 1.0E--5   \\ [0.7ex]
& $\Vert \cdot\Vert_{\bm U_h}$ & 5.6E--3 & 5.6E--3 & 5.6E--3 & 5.6E--3   \\[0.7ex] \hline
\multirow{3}{*}{$\displaystyle \frac{1}{256}$} & $\Vert \cdot\Vert_{\bm P}$ & 6.6E--3 & 6.6E--4 & 6.6E--4 & 6.6E--4   \\ [0.7ex]
& $\Vert \cdot\Vert_{\bm V}$ & 5.1E--6 & 5.1E--6 & 5.1E--6 & 5.1E--6   \\[0.7ex]
& $\Vert \cdot\Vert_{\bm U_h}$ & 2.8E--3 & 2.8E--3 & 2.8E--3 & 2.8E--3
\end{tabular}
\end{center}
\end{table}

\begin{table}[h!]
\caption{Preconditioned MinRes convergence history for solving the Biot problem.}\label{tab4}
\begin{center}
\begin{tabular}{c|c|c|rr|rr|rr|rr|rr|rr}
& & & \multicolumn{10}{c}{$R_1^{-1}$} \\[1ex]\hline
$h$ & $\alpha_p$ & $\lambda$ & \multicolumn{2}{c}{1E0} &  \multicolumn{2}{c}{1E2} &  \multicolumn{2}{c}{1E3} &  \multicolumn{2}{c}{1E4} &  \multicolumn{2}{c}{1E8} &  \multicolumn{2}{c}{1E16} \\ [0.1ex] \hline 
\multirow{12}{*}{$\displaystyle \frac{1}{16}$} & 
\multirow{3}{*}{1E0}  & 1E0 & 19 & 0.37 & 27 & 0.50 & 26 & 0.49 & 19 & 0.38 & 13 & 0.24 &  13 & 0.24 \\[0.1ex]
                               & & 1E4  & 10 & 0.15 & 20 & 0.39 & 19 & 0.38 & 13 & 0.23 & 4 & $<$0.01 & 3 & $<$0.01 \\[0.1ex]
                               & & 1E8 & 10 &  0.11 & 20 & 0.39 & 19 & 0.38 & 13 & 0.23 & 4 & $<$0.01  & 3 & $<$0.01 \\ [0.1ex]\cline{2-15}
& \multirow{3}{*}{1E-4}  & 1E0 & 19 & 0.38 & 35 & 0.58 & 43 & 0.65 & 34 & 0.50 & 19 & 0.29  & 19 & 0.36 \\[0.1ex]
                                    & & 1E4  & 9 & 0.08 & 10 & 0.11 & 12 & 0.17 & 13 & 0.23 & 17 & 0.31  & 5& 0.01 \\[0.1ex]
                                    & & 1E8 & 6 & 0.05 & 8 & 0.07 & 9 & 0.10 & 10 & 0.14 & 11 & 0.18  & 3 & $<$0.01 \\[0.1ex] \cline{2-15}
& \multirow{3}{*}{1E-8}  & 1E0 & 19 & 0.38 & 35 & 0.58 & 43 & 0.65 & 34 & 0.50 & 19 & 0.29  & 19& 0.36 \\[0.1ex]
                                   & & 1E4  & 9 & 0.08 & 10 & 0.11 & 12 & 0.17 & 14 & 0.23 & 20 & 0.31  & 5& 0.01 \\[0.1ex]
                                  & & 1E8 & 6 & 0.05 & 8 & 0.07 & 8 & 0.07 & 9 & 0.08 & 13 & 0.24  & 4& 0.01 \\[0.1ex] \cline{2-15}
& \multirow{3}{*}{0}  & 1E0 & 19 & 0.38 & 35 & 0.58 & 43 & 0.65 & 34 & 0.50 & 19 & 0.29  & 19& 0.36 \\[0.1ex]
                                   & & 1E4  & 9 & 0.08 & 10 & 0.11 & 12 & 0.17 & 14 & 0.23 & 20 & 0.31  & 5 & 0.01 \\[0.1ex]
                                   & & 1E8 & 6 & 0.05 & 8 & 0.07 & 8 & 0.07 & 9 & 0.08 & 13 & 0.24 & 4 & 0.01 \\[0.1ex] \hline
\multirow{12}{*}{$\displaystyle \frac{1}{64}$} &                                   
                                   \multirow{3}{*}{1E0}  & 1E0 & 18 & 0.35 & 27 & 0.49 & 28 & 0.51 & 25 & 0.47 & 12 & 0.20  & 12 & 0.20 \\[0.1ex]
                               & & 1E4  & 9 & 0.12 & 19 & 0.36 & 20 & 0.39 & 16 & 0.30 & 4 & $<$0.01 &  3 & $<$0.01 \\[0.1ex]
                               & & 1E8 & 8 & 0.09 & 19 & 0.36 & 20 & 0.39 & 16 & 0.30 & 4 & $<$0.01 &  3 & $<$0.01 \\[0.1ex] \cline{2-15}
& \multirow{3}{*}{1E-4}  & 1E0 & 19 & 0.36 & 34 & 0.57 & 46 & 0.66 & 47 & 0.61 & 20& 0.39 &  19 & 0.37 \\[0.1ex]
                                    & & 1E4  & 8 & 0.09 & 10 & 0.11 & 12 & 0.17 & 13 & 0.21 & 21 & 0.40 & 5 & 0.01 \\ [0.1ex]
                                    & & 1E8 & 6 & 0.03 & 7 & 0.06 & 8 & 0.09 & 9 & 0.12 & 14 & 0.26 & 3 & $<$0.01 \\[0.1ex] \cline{2-15}
& \multirow{3}{*}{1E-8}  & 1E0 & 19 & 0.36 & 34 & 0.57 & 46 & 0.66 & 47 & 0.61 & 20 & 0.39  & 19 & 0.37 \\ [0.1ex]
                                   & & 1E4  & 8 & 0.09 & 10 & 0.11 & 12 & 0.17 & 13 & 0.21 & 26 & 0.49 & 5 & 0.01 \\[0.1ex]
                                  & & 1E8 & 6 & 0.03 & 7 & 0.06 & 7 & 0.06 & 8 & 0.09 & 13 & 0.22 &  4 & 0.01 \\[0.1ex] \cline{2-15}
& \multirow{3}{*}{0}  & 1E0 & 19 & 0.36 & 34 & 0.57 & 46 & 0.66 & 47 & 0.61 & 20 & 0.39  & 19 & 0.37 \\[0.1ex]
                                   & & 1E4  & 8 & 0.09 & 10 & 0.11 & 12 & 0.17 & 13 & 0.21 & 26 & 0.49  & 5 & 0.01 \\[0.1ex]
                                   & & 1E8 & 6 & 0.03 & 7 & 0.06 & 7 & 0.06 & 8 & 0.09 & 13 & 0.22  & 4 & 0.01 \\[0.1ex] \hline
\multirow{12}{*}{$\displaystyle \frac{1}{256}$} &                                   
                                   \multirow{3}{*}{1E0}  & 1E0 & 18 & 0.34 & 27 & 0.49 & 28 & 0.51 & 25 & 0.49 &  12 & 0.20 & 12 & 0.20 \\[0.1ex]
                               & & 1E4  & 9 & 0.11 & 19 & 0.36 & 20 & 0.39 & 16 & 0.31 & 4 & $<$0.01 &  3 & $<$0.01 \\[0.1ex]
                               & & 1E8 & 9 & 0.11 & 19 & 0.36 & 20 & 0.39 & 16 & 0.31 & 4 & $<$0.01 &  3 & $<$0.01 \\[0.1ex] \cline{2-15}
& \multirow{3}{*}{1E-4}  & 1E0 & 19 & 0.34 & 32 & 0.56 & 44 & 0.66 & 47 & 0.67 & 22& 0.45 &  21 & 0.37 \\[0.1ex]
                                    & & 1E4  & 8 & 0.08 & 9 & 0.11 & 11 & 0.19 &13 & 0.21 & 20 & 0.40 & 5 & 0.01 \\[0.1ex]
                                    & & 1E8 & 6 & 0.03 & 7 & 0.05 & 8 & 0.08 & 9 & 0.11 & 14 & 0.26 &  4 & 0.01 \\ [0.1ex]\cline{2-15}
& \multirow{3}{*}{1E-8}  & 1E0 & 19 & 0.34 & 32 & 0.56 & 44 & 0.66 & 47 & 0.67 & 22 & 0.40  & 21 & 0.37 \\ [0.1ex]
                                   & & 1E4  & 8 & 0.08 & 9  & 0.11 & 11 & 0.19 &13 & 0.21 & 26 & 0.49 &  5 & 0.01 \\ [0.1ex]
                                  & & 1E8 & 6 & 0.03 & 7 & 0.05 &  8 & 0.08  & 8 & 0.08 & 12 & 0.20 &  4 & 0.01 \\ [0.1ex] \cline{2-15}
& \multirow{3}{*}{0}  & 1E0 & 19 & 0.34 & 32 & 0.56 & 44 & 0.66 & 47 & 0.67 & 22 & 0.40 & 21 & 0.37 \\ [0.1ex]
                                   & & 1E4  & 8 & 0.08 & 10 & 0.11 & 11 & 0.19 & 13 & 0.21 & 26 & 0.49  & 5 & 0.01 \\ [0.1ex]
                                   & & 1E8 & 6 & 0.03 & 7 & 0.05 & 8 & 0.08 & 8 & 0.08 & 12 & 0.20 & 4 & 0.01                                
\end{tabular}
\end{center}
\end{table}
As can be seen from Tables~\ref{tab1}--\ref{tab3} the error in the considered parameter-dependent norms decreases by a factor $2$ when
decreasing the mesh size by the same factor independently of the model parameters.
The results in Table~\ref{tab4} suggest that the number of MinRes iterations required to achieve a prescribed solution accuracy is bounded by a constant 
independent of $\lambda$, $R_1^{-1}$, $\alpha_{p_1}$ and $h$ while the average residual reduction factor always remains smaller than $0.70$. Note 
that in this table the authors have tried to present the most unfavourable setting of input parameters in order to stress test the proposed method.

\subsection{The two-network model}

The governing partial differential equations of the Biot-Barenblatt model in which the flux-based MPET system involves two
pressures and two fluxes are given by

\begin{subequations}
\begin{align}
-{\rm div} (\bm \sigma-p_1\bm I-p_2\bm I ) &= \bm f ,\\
R^{-1}_i\bm v_i + \nabla p_i &= 0, \ \quad i=1,2, \\
%R^{-1}_2\bm v_2 + \nabla p_2 &= 0\\
-{\rm div} \bm u - {\rm div}\bm v_i -\alpha_{p_i}p_i + \sum_{\substack{j=1\\j\neq i}}^2\alpha_{ij} p_j &= g_i, \quad i=1,2 .
%-\text{ div } \bm u - \text{ div }\bm v_2 + \alpha_{21} p_1 - \alpha_{p_2}p_2 &= g_2
\end{align}
\label{eq:beweak}
\end{subequations}

\noindent
We consider here the cantilever bracket benchmark problem proposed by the National Agency for Finite Element Methods and
Standards in~\cite{NAFEMS1990} with $\boldsymbol f=0$, $g_1=0$ and $g_2=0$.\\

The boundary of the domain $\Omega=[0,1]^2$ is split into $\Gamma_1$, $\Gamma_2$, $\Gamma_3$ and $\Gamma_4$
denoting the bottom, right, top and left boundaries respectively, and  the boundary conditions $\boldsymbol u=0$ on $\Gamma_4$,
$(\boldsymbol {\sigma} -p_1\boldsymbol I-p_2\boldsymbol I)\boldsymbol n=(0,0)^T$ on $\Gamma_1\cup\Gamma_2$, 
$(\boldsymbol {\sigma} -p_1\boldsymbol I-p_2\boldsymbol I)\boldsymbol n=(0,-1)^T$ on $\Gamma_3$, $p_1=2$ on
$\Gamma$ $p_2=20$ on $\Gamma$ are imposed.

The base values of the model parameters are taken from~\cite{Kolesov2017} and are presented in Table~\ref{parameters_Barenblatt}. 
The computed numerical results in Table~\ref{tab:5} show robust behaviour with respect to mesh refinements and variation of the parameters
including high contrasts of the hydraulic conductivities. 
The parameter $K_2$ has been varied over a wider range than $K_1$ as it appeared to be the more interesting case.
\begin{table}[h!]
\caption{Base values of model parameters for the Barenblatt model.}
\label{parameters_Barenblatt}
\begin{center}
\begin{tabular}{ccc}
\hline 
\cellcolor{gray!35} parameter & \cellcolor{gray!20} value & \cellcolor{gray!5} unit \\ \hline
$\lambda$ & $4.2$ & MPa \\
$\mu$ & $2.4$ & MPa \\
$c_{p_1}$ & $54$ & (GPa)$^{-1}$ \\
$c_{p_2}$ & $14$ & (GPa)$^{-1}$\\
$\alpha_{1}$ & $0.95$ & \\
$\alpha_{2}$ & $0.12$ & \\
\multirow{2}{*}{$\beta$} & $5$ & $10^{-10}$kg/(m$\cdot$s) \\
& $100$ & $10^{-10}$kg/(m$\cdot$s) \\
$K_1$ & $6.18$ & $10^{-15}$m$^2$ \\
$K_2$ & $27.2$ & $10^{-15}$m$^2$
\end{tabular}
\end{center}
\end{table}
\begin{table}[hbtp!]
\caption{Preconditioned MinRes convergence history for solving the Barenblatt problem.}
%system up to relative residual accuracy $10^{-8}$.}
\label{tab:5}
\begin{center}
\begin{tabular}{c|c||c|rr|rr|rr}
$h$ & $\beta$ &   & \multicolumn{2}{c}{$K_1\cdot 10^{-2}$} &  \multicolumn{2}{c}{$K_1\cdot 10^{-1}$} &  \multicolumn{2}{c}{$K_1$}  \\ \hline \hline
\multirow{8}{*}{$\displaystyle \frac{1}{16}$} & 
\multirow{4}{*}{5E--10}  & $K_2$ & 16 & 0.31 & 16 & 0.31 & 16 & 0.31  \\
                               & & $K_2\cdot 10^{2}$ & 21 & 0.41 & 21 & 0.41 & 21 & 0.41  \\
                               & & $K_2\cdot 10^{4}$ & 37 & 0.61 & 37 & 0.61 & 37 & 0.61 \\
                                & & $K_2\cdot 10^{6}$ & 29 & 0.51 & 29 & 0.51 & 29 & 0.51 \\ \cline{2-9}
& \multirow{4}{*}{1E-8}  & $K_2$ & 16 & 0.31 & 16 & 0.31 & 16 & 0.31   \\
                                     & & $K_2\cdot 10^{2}$ & 21 & 0.41 & 21 & 0.41 & 21 & 0.41  \\ 
                                    & & $K_2\cdot 10^{4}$ & 37 & 0.61 & 37 & 0.61 & 37 & 0.61 \\
                                    & & $K_2\cdot 10^{6}$ & 29 & 0.51 & 29 & 0.51 & 29 & 0.51  \\ \hline
\multirow{8}{*}{$\displaystyle \frac{1}{64}$} &                                   
                                   \multirow{4}{*}{5E--10}  & $K_2$ & 18 & 0.33 & 18 & 0.33 & 18 & 0.33 \\
                               & & $K_2\cdot 10^{2}$ & 32 & 0.55 & 32 & 0.55 & 32 & 0.55  \\
                                & & $K_2\cdot 10^{4}$ & 38 & 0.61 & 38 & 0.61 & 38 & 0.61  \\ 
                                 & & $K_2\cdot 10^{6}$ & 27 & 0.49 & 27 & 0.49 & 27 & 0.49  \\  \cline{2-9}
& \multirow{4}{*}{1E-8}  & $K_2$ & 18 & 0.33 & 18 & 0.33 & 18 & 0.33  \\
                                   & & $K_2\cdot 10^{2}$ & 32 & 0.55 & 32 & 0.55 & 32 & 0.55  \\ 
                                    & & $K_2\cdot 10^{4}$ & 38 & 0.61 & 38 & 0.61 & 38 & 0.61  \\ 
                                    & & $K_2\cdot 10^{6}$ & 27 & 0.49 & 27 & 0.49 & 27 & 0.49  \\ \hline
\multirow{8}{*}{$\displaystyle \frac{1}{256}$} &                                   
                                   \multirow{4}{*}{5E--10}  & $K_2$ & 22 & 0.43 & 22 & 0.43 & 22 & 0.43  \\ 
                               & & $K_2\cdot 10^{2}$ & 35 & 0.58 & 35 & 0.58 & 35 & 0.58  \\ 
                               & & $K_2\cdot 10^{4}$ & 37 & 0.60 & 37 & 0.60 & 37 & 0.60  \\ 
                              & & $K_2\cdot 10^{6}$ & 27 & 0.48 & 27 & 0.48 & 27 & 0.48  \\ \cline{2-9}
& \multirow{4}{*}{1E--8}  & $K_2$ & 22 & 0.43 & 22 & 0.43 & 22 & 0.43  \\
                                    & & $K_2\cdot 10^{2}$ & 35 & 0.58 & 35 & 0.58 & 35 & 0.58  \\  
                                     & & $K_2\cdot 10^{4}$ & 37 & 0.60 & 37 & 0.60 & 37 & 0.60  \\ 
                                      & & $K_2\cdot 10^{6}$ & 27 & 0.48 & 27 & 0.48 & 27 & 0.48                               
\end{tabular}
\end{center}
\end{table}

\subsection{The four-network problem}

In this example we consider the four-network MPET problem.
The boundary of $\Omega$ is split into four non-overlapping parts $\Gamma_1$, $\Gamma_2$, 
$\Gamma_3$ and $\Gamma_4$ in the same manner as for the {Barenblatt model} and we set 
$\boldsymbol u=0$~on~$\Gamma_4$, $(\boldsymbol {\sigma} -p_1\boldsymbol I-p_2\boldsymbol I -p_3\boldsymbol I-p_4\boldsymbol I)\boldsymbol n=(0,0)^T$ on $\Gamma_1\cup\Gamma_2$, 
$(\boldsymbol {\sigma} -p_1\boldsymbol I-p_2\boldsymbol I -p_3\boldsymbol I-p_4\boldsymbol I)\boldsymbol n=(0,-1)^T$ on $\Gamma_3$, 
$p_1=2$~on~$\Gamma$, $p_2=20$ on $\Gamma$, $p_3=30$ on $\Gamma$ and $p_4=40$ on $\Gamma$ . 
The right hand sides in~\eqref{eq:6} are chosen to be 
$\boldsymbol f=0$, $g_1=0$, $g_2=0$, $g_3=0$ and $g_4=0$. \\

The base values of the parameters for numerical testing are given in Table~\ref{parameters_MPET4}
and taken from~\cite{Vardakis2016investigating} where the four-network MPET model has been used to
simulate fluid flow in the human brain. 
Table~\ref{tab7} shows robust behaviour of the proposed block-diagonal preconditioner in~\eqref{Preconditioner:Bh}
as the number of MinRes iterations and the average residual reduction factor remain  uniformly bounded for large variations
of the coefficients $\lambda$, $K_3$ and $K=K_1=K_2=K_4$.  

Here, it is important to note that the authors have attempted to present again the least
optimal choice of parameters for testing their implementation.
{
\vspace{2ex}
\begin{table}[h!]
\caption{Base values of model parameters for the four-network MPET model.}
\label{parameters_MPET4}
\begin{center}
\begin{tabular}{ccc}
\hline 
\cellcolor{gray!35} parameter & \cellcolor{gray!20} value & \cellcolor{gray!5} unit \\[0.4ex] \hline
$\lambda$ & $505$ & Nm$^{-2}$ \\ [0.4ex]
$\mu$ & $216$ & Nm$^{-2}$ \\ [0.4ex]
$c_{p_1}=c_{p_2}=c_{p_3}=c_{p_4}$ & $4.5\cdot10^{-10}$ & m$^2$N$^{-1}$ \\ [0.4ex]
$\alpha_{1}=\alpha_{2}=\alpha_{3}=\alpha_{4}$  & $0.99$ &\\ [0.4ex]
$\beta_{12}=\beta_{24}$ & $1.5\cdot 10^{-19}$ & m$^2$N$^{-1}$s$^{-1}$ \\ [0.4ex]
$\beta_{23}$ & $2.0\cdot 10^{-19}$ & m$^2$N$^{-1}$s$^{-1}$ \\[0.4ex]
$\beta_{34}$ & $1.0\cdot 10^{-13}$ & m$^2$N$^{-1}$s$^{-1}$ \\ [0.4ex]
$K_1=K_2=K_4=K$ & $(1.0 \cdot 10^{-10})/(2.67\cdot10^{-3})$ & m$^2/$Nsm$^{-2}$ \\ [0.4ex]
$K_3$ & $(1.4 \cdot 10^{-14})/(8.9\cdot10^{-4})$ & m$^2/$Nsm$^{-2}$
\end{tabular}
\end{center}
\end{table}
}

\begin{table}[htb!]
\caption{Preconditioned MinRes convergence history for solving the four-network MPET problem.}
\label{tab7}
\begin{center}
\begin{tabular}{c|c|c|rr|rr|rr|rr|rr|rr}
$h$ &  &  & \multicolumn{2}{c}{$K_3\cdot 10^{-2}$} &  \multicolumn{2}{c}{$K_3$} &  \multicolumn{2}{c}{$K_3\cdot 10^{2}$} &  \multicolumn{2}{c}{$K_3\cdot 10^{4}$} & 
 \multicolumn{2}{c}{$K_3\cdot 10^{6}$} &  \multicolumn{2}{c}{$K_3\cdot 10^{10}$} \\ [0.1ex] \hline 
\multirow{9}{*}{$\displaystyle \frac{1}{32}$} & 
\multirow{3}{*}{$\lambda$}  & $K\cdot 10^{-2}$ & 34 & 0.56 & 32 & 0.56 & 26 & 0.47 & 23 & 0.42 & 19 & 0.37 &  19 & 0.37 \\ [0.1ex]
                                & & $K$  & 24 & 0.48 & 24 & 0.49 & 24 & 0.49 & 22 & 0.42 & 21 & 0.41 & 20 & 0.40 \\[0.1ex]
                               & & $K\cdot 10^{2}$ & 21 &  0.41 & 21 & 0.41 & 21 & 0.41 & 26 & 0.49 & 41 & 0.63  & 39 & 0.62 \\ [0.1ex] \cline{2-15}
& \multirow{3}{*}{$\lambda\cdot 10^4$}  & $K\cdot 10^{-2}$ & 18 & 0.35 & 25 & 0.48 & 30 & 0.53 & 34 & 0.57 & 34 & 0.57 &  34 & 0.57 \\ [0.1ex]
                                    & & $K$  & 12 & 0.20 & 20 & 0.40 & 35 & 0.59 & 31 & 0.54 & 31 & 0.54 & 31 & 0.54 \\ [0.1ex]
                                    & & $K\cdot 10^{2}$ & 9 & 0.12 & 18 & 0.40 & 34 & 0.58 & 21 & 0.41 & 14 & 0.26  & 14 & 0.26 \\ [0.1ex] \cline{2-15}
& \multirow{3}{*}{$\lambda\cdot 10^8$}  & $K\cdot 10^{-2}$ & 14 & 0.25 & 14 & 0.27 & 12 & 0.19 & 12 & 0.20 & 12 & 0.20  & 12 & 0.20 \\ [0.1ex]
                                   & & $K$  & 12 & 0.20 & 14 & 0.26 & 9 & 0.12 & 7 & 0.07 & 7 & 0.07  & 7 & 0.07 \\ [0.1ex]
                                  & & $K\cdot 10^{2}$ & 11 & 0.18 & 14 & 0.26 & 9 & 0.12 & 6 & 0.04 & 5 & 0.02  & 5 & 0.02  \\ [0.1ex] \hline
\multirow{9}{*}{$\displaystyle \frac{1}{64}$} &                                   
                                   \multirow{3}{*}{$\lambda$}  & $K\cdot 10^{-2}$ & 34 & 0.56 & 32 & 0.56 & 26 & 0.47 & 21 & 0.41 & 19 & 0.37 &  19 & 0.37 \\ [0.1ex]
                               & & $K$  & 24 & 0.48 & 24 & 0.49 & 24 & 0.49 & 23 & 0.42 & 22 & 0.42 & 21 & 0.41 \\ [0.1ex]
                               & & $K\cdot 10^{2}$ & 21 & 0.41 & 21 & 0.41 & 21 & 0.41 & 36 & 0.58 & 45 & 0.66  & 45 & 0.66 \\ [0.1ex] \cline{2-15}
& \multirow{3}{*}{$\lambda\cdot 10^4$}  & $K\cdot 10^{-2}$ & 20 & 0.40 & 28 & 0.51 & 34 & 0.58 & 34 & 0.57 & 34 & 0.57 &  34 & 0.57 \\ [0.1ex]
                                    & & $K$  & 13 & 0.22 & 25 & 0.48 & 36 & 0.60 & 31 & 0.54 & 31 & 0.54 & 31 & 0.54 \\ [0.1ex]
                                    & & $K\cdot 10^{2}$ & 6 & 0.03 & 25 & 0.46 & 36 & 0.60 & 21 & 0.41 & 14 & 0.26  & 14 & 0.26 \\ \cline{2-15}
& \multirow{3}{*}{$\lambda\cdot 10^8$}  & $K\cdot 10^{-2}$ & 14 & 0.25 & 14 & 0.27 & 12 & 0.19 & 12 & 0.20 & 12 & 0.20  & 12 & 0.20 \\
                                   & & $K$  & 12 & 0.20 & 14 & 0.26 & 9 & 0.12 & 7 & 0.07 & 7 & 0.07 & 7 & 0.07 \\ [0.1ex]
                                  & & $K\cdot 10^{2}$ & 12 & 0.20 & 14 & 0.26 & 9 & 0.12 & 6 & 0.04 & 5 & 0.02 &  5 & 0.02  \\ [0.1ex] \hline
\multirow{9}{*}{$\displaystyle \frac{1}{128}$} &                                   
                                   \multirow{3}{*}{$\lambda$}  & $K\cdot 10^{-2}$ & 34 & 0.55 & 32 & 0.56 & 26 & 0.47 & 21 & 0.41 & 19 & 0.37 &  19 & 0.37 \\ [0.1ex]
                               & & $K$  & 24 & 0.48 & 24 & 0.49 & 24 & 0.49 & 24 & 0.44 & 23 & 0.42 & 22 & 0.41 \\ [0.1ex]
                               & & $K\cdot 10^{2}$ & 21 & 0.41 & 21 & 0.41 & 21 & 0.41 & 43 & 0.64 & 49 & 0.68  & 49 & 0.68 \\[0.1ex] \cline{2-15}
& \multirow{3}{*}{$\lambda\cdot 10^4$}  & $K\cdot 10^{-2}$ & 22 & 0.41 & 30 & 0.55 & 35 & 0.59 & 34 & 0.57 & 34 & 0.57 &  34 & 0.57 \\ [0.1ex]
                                    & & $K$  & 14 & 0.28 & 29 & 0.54 & 36 & 0.60 & 31 & 0.54 & 31 & 0.54 & 31 & 0.54 \\ [0.1ex]
                                    & & $K\cdot 10^{2}$ & 12 & 0.20 & 30 & 0.54 & 36 & 0.60 & 21 & 0.41 & 14 & 0.26 &  14 & 0.26 \\ [0.1ex] \cline{2-15}
& \multirow{3}{*}{$\lambda\cdot 10^8$}  & $K\cdot 10^{-2}$ & 14 & 0.25 & 14 & 0.27 & 12 & 0.19 & 12 & 0.20 & 12 & 0.20  & 12 & 0.20 \\ [0.1ex]
                                   & & $K$  & 12 & 0.20 & 14  & 0.26 & 9 & 0.12 & 7 & 0.07 & 7 & 0.07 &  7 & 0.07 \\ [0.1ex]
                                  & & $K\cdot 10^{2}$& 12 & 0.20 & 14 & 0.26 &  9 & 0.12  & 6 & 0.04 & 5 & 0.02 &  5 & 0.02                                 
\end{tabular}
\end{center}
\end{table}

\section{Conclusions}\label{conclusion}
In this paper,
motivated  by the approach recently presented by Hong and Kraus~[Parameter-robust stability of classical three-field formulation of Biot's
consolidation model, ETNA (to appear)] for the Biot model,
%of using strongly mass-conservative discretizations for the Biot's consolidation model
we establish the uniform stability, {design} stable disretizations and a parameter-robust preconditioners
for flux-based formulations of multiple-network poroelastic systems. 
Novel proper parameter-matrix-dependent norms that provide the key for establishing uniform inf-sup stability of the continuous
problems are introduced. The stability results that could be obtained using the presented {\it matrix technique} are uniform not
only with respect to the Lam\'e parameter $\lambda$ but also with respect to all the other model parameters such as small or large
permeability coefficients $K_i$, arbitrary small or even vanishing storage coefficients $c_{p_i}$, arbitrary small or even vanishing
network transfer coefficients $\beta_{ij}, i,j=1,\cdots,n$, the scale of the networks $n$, and the time step size $\tau$. 

Moreover, strongly mass conservative and uniformly stable discretizations are proposed and corresponding uniform and
optimal error estimates proved which are also independent of the Lam\'e parameter $\lambda$, the permeability coefficients $K_i$,
the storage coefficients $c_{p_i}$,  the network transfer coefficients $\beta_{ij}, i,j=1,\cdots,n$, the scale  of the networks $n$, the
time step size $\tau$ and the mesh size $h$. The transfer of the canonical (norm-equivalent) operator preconditioners from the
continuous to the discrete level lays the foundation for optimal and fully robust iterative solution methods.
Numerical experiments that are motivated by practical applications are presented confirming both the uniform and optimal
convergence of the proposed finite element methods and the uniform robustness of the norm-equivalent preconditioners.

%\clearpage

\bibliographystyle{plain}
\bibliography{reference_mpet}

\end{document}